\documentclass[a4paper,12pt,reqno]{amsart}
\usepackage{amsfonts}

\usepackage[T1]{fontenc} 
\usepackage[latin1]{inputenc}

\pdfoutput=1

\usepackage{tikz}
\usepackage{latexsym}
\usepackage[colorlinks=true, pdfstartview=FitV, linkcolor=blue, citecolor=blue, urlcolor=blue,pagebackref=false]{hyperref}

\usepackage{amsmath, amsthm, amssymb, color, graphicx, hyperref, mathrsfs}

\usepackage[margin=1.3in,marginparwidth=1.5cm, marginparsep=0.5cm]{geometry}

\usepackage{booktabs} 

\usepackage{microtype}

\usepackage{mathrsfs}

\usepackage{marginnote}
\usepackage{comment}

\definecolor{darkgreen}{rgb}{0,0.5,0}
\definecolor{darkred}{rgb}{0.7,0,0}
\definecolor{darkblue}{rgb}{0,0,0.7}
\newcommand{\margin}[1]{\textcolor{magenta}{*}\marginpar[\textcolor{magenta} {  \raggedleft  \footnotesize  #1 }  ]{ \textcolor{magenta} { \raggedright  \footnotesize  #1 }  }}
\newcommand{\note}[2]{ \hskip 2cm  \textcolor{blue}{\large \bf #1 :}   \vline\,\vline \hskip 0.5 cm \parbox{10 cm}{ #2}  }
\renewcommand{\margin}[1]{}
\renewcommand{\note}[2]{}

\def\rm{\reversemarginpar}

\theoremstyle{plain}
\newtheorem{theorem}{Theorem}[section]
\newtheorem{corollary}[theorem]{Corollary}
\newtheorem{lemma}[theorem]{Lemma}
\newtheorem{proposition}[theorem]{Proposition}

\newtheorem{definition}[theorem]{Definition}

\newtheorem{question}[theorem]{Question}

\theoremstyle{remark}
\newtheorem{remark}[theorem]{Remark}

\numberwithin{equation}{section}
\numberwithin{table}{section}

\newcommand{\N}{\mathbb{N}}
\newcommand{\R}{\mathbb{R}}
\newcommand{\C}{\mathbb{C}}

\newcommand{\Z}{\mathbb{Z}}

\newcommand{\E}{\mathbb{E}}
\newcommand{\T}{\mathbb{T}}

\renewcommand{\C}{\mathbb{C}}
\renewcommand{\Z}{\mathbb{Z}}
\renewcommand{\N}{\mathbb{N}}
\renewcommand{\S}{\mathbb{S}}
\def\H{\mathbb{H}}

\def\T{\mathbb{T}}

\newcommand{\QED}{\hfill \ensuremath{\Box}}
\let\qed=\QED

\def\calC{\mathcal{C}}

\def\calH{\mathcal{H}}

\def\calQ{\mathcal{Q}}

\def\calT{\mathcal{T}}

\def\SLE{\mathrm{SLE}}

\def\Var#1{\mathrm{Var}\bigl[ #1\bigr]}

\def\dist{\mathrm{dist}}
\def\ceil#1{\lceil{#1}\rceil}
\def\floor#1{\lfloor{#1}\rfloor}

\def\P{\mathbb{P}} 
\def\E{\mathbb{E}} 
\def\md{\mid}

\def \eps {\epsilon}

\def\Bb#1#2{{\def\md{\bigm| }#1\bigl[#2\bigr]}}

\def\Pb{\Bb\P}
\def\Eb{\Bb\E}

\def\FK#1#2#3{{\def\md{\bigm| } \P_{#1}^{\,#2}  \bigl[  #3 \bigr]}}

\def \p {{\partial}}
\def\closure{\overline}

\def\scal#1{\langle#1\rangle}

\renewcommand{\subset}{\subseteq}

\renewcommand{\hat}{\widehat}

\def\<#1{\langle #1\rangle}
\newcommand{\red}[1]{\textcolor{red}{#1}}

\newcommand{\purple}[1]{\textcolor{purple}{#1}}

\def\nn{\nonumber}
\def\bi{\begin{itemize}}  
\def\ei{\end{itemize}}
\def\bnum{\begin{enumerate}} 
\def\enum{\end{enumerate}}
\def\ni{\noindent}
\def\bf{\bfseries}

\usepackage{mhequ} 

\colorlet{symbols}{blue!90!black}
\colorlet{testcolor}{green!60!black}

\def\1{\mathbf{{1}}}

\def\Wick#1{\,\colon\! #1 \colon}

\def\s{\mathfrak{s}}

\def\g2{\frac {\gamma^2} 2}
\newcommand{\Cs}{\mathcal{C}_{\mathfrak{s}}}
 
\def\GFF{\mathrm{GFF}}

\title{Dynamical Liouville}

\author{Christophe Garban}
\address[Christophe Garban]
{Univ Lyon, Universit\'e Claude Bernard Lyon 1, CNRS UMR 5208, Institut Camille Jordan, 69622 Villeurbanne, France}
\email{garban@math.univ-lyon1.fr}


\begin{document}


\begin{abstract}
The aim of this paper is to analyze an SPDE which arises naturally in the context of Liouville quantum gravity.
This SPDE shares some common features with the so-called Sine-Gordon equation and is built to preserve the {\em Liouville measure} which has been constructed recently on the two-dimensional sphere $\S^2$ and the torus $\T^2$ in the work by David-Kupiainen-Rhodes-Vargas \cite{sphere, torus}.
The SPDE we shall focus on has the following (simplified) form:
\[
\p_t X = \frac 1 {4\pi} \Delta X -  e^{\gamma X}  + \xi\,,
\]
where $\xi$ is a space-time white noise on $\R_+ \times \S^2$ or $\R_+ \times \T^2$. The main aspect which distinguishes this singular stochastic SPDE with well-known SPDEs studied recently (KPZ, dynamical $\Phi^4_3$, dynamical Sine-Gordon, generalized KPZ, etc.) is the presence of intermittency. One way of picturing this effect is that a naive rescaling argument suggests the above SPDE is sub-critical for all $\gamma>0$, while we don't expect solutions to exist when $\gamma > 2$. 
In this work, we initiate the study of this intermittent SPDE by analyzing what one might call the ``classical'' or ``Da Prato-Debussche'' phase which corresponds here to $\gamma\in[0,\gamma_{dPD}=2\sqrt{2} -\sqrt{6})$. 
By exploiting the positivity of the non-linearity $e^{\gamma X}$, we can push this classical threshold further and obtain this way a weaker notion of solution when $\gamma\in[\gamma_{dPD}, \gamma_{\mathrm{pos}}=2\sqrt{2} - 2)$.  Our proof requires an analysis of the Besov regularity of natural space/time Gaussian multiplicative chaos  (GMC) measures. 
Regularity Structures of arbitrary high degree should potentially give strong solutions all the way to the same threshold $\gamma_{\mathrm{pos}}$ and should not push this threshold further unless the notion of regularity is suitably adapted to the present intermittent situation.
Of independent interest, we prove along the way (using techniques from \cite{sine}) a stronger convergence result for approximate GMC measures $\mu_\eps \to \mu$ which holds in Besov spaces.
\end{abstract}
\maketitle

\setcounter{tocdepth}{1}
\tableofcontents

\section{Introduction} 

The SPDE we shall focus on is motivated by Liouville Quantum Gravity. We start with some background on Liouville Quantum Gravity before stating our main results. 

\subsection{Liouville measure and uniformisation of planar maps}

$ $

Duplantier and Sheffield conjectured in their seminal paper \cite{DS} that the conformal embedding of {\em planar maps} with the topology of the sphere $\S^2$ or the torus $\T^2$ is described asymptotically (as the number of faces tends to infinity) by \textbf{multiplicative chaos measures} of the form 
\[
e^{\gamma X} dx dy
\]
where $X$ is a Gaussian Free Field on the sphere $\S^2$ (or $\T^2$) (with, say, average 0) and $\gamma$ depends on the universality class of the chosen planar map model. One has the useful correspondance $\gamma \equiv \sqrt{\kappa}$ if the planar map is naturally associated with $\SLE_\kappa$ processes. See \cite{DS, bourbaki, LecturesLQG, Oxford4, Cargese, VincentDOZZ} for background on this embedding conjecture and more generally on Liouville quantum gravity. 
\medskip

In order to write down a precise conjecture, several difficulties need to be addressed.
\bnum
\item Since the Gaussian Free Field $X$ is a highly oscillating random distribution, one first has to precise what is the meaning of $e^{\gamma X} dx dy$. This part is now standard and is known as \textbf{multiplicative chaos theory.} It goes back to the work of Kahane \cite{Kahane}. See \cite{ChaosReview, Oxford4, JuhanLectures} and references therein for useful reviews on this theory.   
\item A more serious difficulty has to do with the appropriate choice of normalisation. 
Planar maps are naturally equipped with a \textbf{probability measure} which assigns mass $1/n$ on each of the $n$ faces of the planar map. The asymptotic random measure should then be a random probability measure on the sphere or the torus. Now, as the sphere and the torus have no boundary, the Gaussian Free Field $X$ is only well-defined up to additive constants. One convenient way to proceed is to fix that additive constant so that $\int_{\S^2} X =0$, say. If one proceeds this way, it is easy to check that the multiplicative chaos measure $e^{\gamma X} dx dy$ as constructed in \cite{Kahane} will have some random positive mass. To make the conjecture precise, one should then {\em renormalize} $e^{\gamma X} dx dy$ in some way in order to obtain a probability measure. Two natural ideas first come to mind: either consider $e^{\gamma X} / \bigl(\int_{\S^2} e^{\gamma X} dxdy \bigr)$ or condition the free field $X$ so that $e^{\gamma X} dxdy$ is a probability measure on $\S^2$ (same thing on $\T^2$). As we shall see below, the right candidate does not correspond to either of these two naïve procedures. 
\item The last difficulty before settling  a precise conjecture has to do with Riemann uniformisation. Here one needs to distinguish the cases of the sphere and the torus. 
\bi
\item[a)] A conformal embedding in $\S^2$ (using either circle packings, Riemann uniformisation or any other natural choice) is rigid up the Möbius transformations $\S^2 \to \S^2$. To break this Möbius invariance which has three complex degrees of freedom (6 real degrees of freedom), the easiest way is to consider maps with three \textbf{marked vertices} $w_1,w_2,w_3$ and to add the additional constraint that the conformal embedding should send these marked vertices to three prescribed points on the sphere $x_1,x_2,x_3$.

\item[b)] In the case of the torus $\T^2$, the situation is different as there are also several possible ways of equipping $\T^2$ with a conformal structure. 
The possible conformal structures are given by the Modular space which in the case of $\T^2$ is non-trivial and is parametrised by a subset of $\C$ (any fundamental domain  for the action of $\mathrm{PSL}_2(\Z)$ on the upper-half plane $\H$). We refer to \cite{torus} for background and notations. Once the choice $\tau\in \H$ of conformal structure on $\T^2$ is prescribed, the conformal embedding is unique up to Möbius transformations $\T^2\to \T^2$ which have only one complex degree of freedom. As such, it is enough to consider planar maps on $\T$ with one marked vertex $w_1$ which is mapped to some $x_1\in \T^2$. 
\ei
\enum

Summarising the above discussion, a precise conjecture for the conformal embedding of large planar maps $\simeq \S^2$ or $\T^2$ should consist of a random \textbf{probability measure} on $\S^2$ (resp. $\T^2$) which is locally of the form $e^{\gamma X} dx dy$ and globally parametrized by three prescribed points $x_1,x_2,x_3 \in \S^2$ (resp. a conformal structure $\tau\in \H$ and one prescribed point $z\in \T^2$). Such a random probability measure on $\S^2$ was successfully built in the work \cite{sphere} and the case of $\T^2$ was solved in \cite{torus}. Each of these measures will naturally be called from then on \textbf{Liouville measures}. 
See the works \cite{disk, genus, Annulus} for extensions of this construction to other topologies (disk, genus $g\geq 2$, annuli)  and the works \cite{Seiberg, Ward, DOZZ} for striking recent results on the structure and properties of the Liouville field (DOZZ formula etc).  
In the case of $\S^2$, a different viewpoint on the Liouville field is provided by the works \cite{zipper,MatingTrees}. These works introduce the so-called \textbf{quantum cones measures} which in some sense describe the similar conformal embedding of planar maps except two out of three complex degrees of freedom are fixed. See the work \cite{Juhan} which highlight a deep and non-trivial connection between the {\em Liouville action approach} and the {\em Mating of Trees} approach.

The measure constructed in \cite{sphere} has an explicit and rather involved Radon-Nikodym derivative w.r.t the random measure $e^{\gamma X} dx dy$ which we shall now describe.

\subsection{Liouville measure on the sphere}

The construction in \cite{sphere} relies on the following \textbf{Liouville action} from which \textbf{Liouville quantum field theory (LQFT)} is built. 
\begin{align}\label{e.LA}
S_L(X,g):= \frac 1 {4\pi} \int_{\S^2} (|\nabla_g X|^2 + Q R_g(z) X(z) + 4\pi \mu e^{\gamma X(z)} ) g(z) dz
\end{align}
where $Q=Q_\gamma = \frac 2 \gamma + \frac \gamma 2$ is an important running constant in Liouville quantum gravity that we shall use throughout. The quadratic term in the action $\frac 1 {4\pi} \int_{\S^2} |\nabla_g X|^2 g(z) dz$ corresponds to a well-known Gaussian process called the \textbf{Gaussian Free Field}. 
\begin{remark}\label{r.shift}
It is important to notice at this point that the Gaussian Free Field appears with a slightly unusual renormalisation (especially in SPDEs): we have $\frac 1 {4\pi} |\nabla X|^2$ in the action instead of the more common $\frac 1 2 |\nabla X|^2$. This will only have the effect of shifting things by a multiplicative constant $\sqrt{2\pi}$.
\end{remark}

Let us introduce the relevant notations. Call $\nu_{\GFF^0}$ the law of a Gaussian Free Field on either $\S^2$ or $\T^2$ with vanishing mean and corresponding to the Hamiltonian $\frac 1 {4\pi} |\nabla X|^2 $. There are several ways to build this process (see for example \cite{Oxford4,JuhanLectures}). With our choice of renormalisation, $\bar X\sim \nu_{\GFF^0}$  has the following simple spectral expansion:
\[
\bar X \overset{(law)}= \sum_{n\geq 0} a_n \sqrt{2\pi} \frac 1 {\sqrt{\lambda_n}} \psi_n(x) 
\]
where $\{a_n\}$ are i.i.d $N(0,1)$ and $\{(\lambda_n,\psi_n)\}_n$ is an orthonormal basis of $L^2(\S^2)$ or $L^2(\T^2)$ which diagonalizes the Laplacian. $\nu_{\GFF^0}$ can be seen as a probability measure on Sobolev spaces of negative index $\calH^{-\alpha}$ for any $\alpha>0$. (Where $\calH^{-\alpha}$ stands for $\calH^{-\alpha}(\S^2)$ in the case of the sphere and resp. for $\T^2$).  
We shall use the following infinite measure on $\R \times \calH^{-\alpha}$: 
\begin{align*}\label{e.GFF}
d\nu_\GFF(c,\bar X) := d\lambda(c)  d\nu_{\GFF^0}(\bar X) 
\end{align*}
where $\lambda$ is the Lebesgue measure on the real-line $\R$.

\begin{definition}[Un-normalized GFF on the sphere/torus]\label{d.GFF}
We define the \textbf{un-normalized Gaussian Free Field} on the sphere/torus to be the distribution 
\[
X = c +\bar X
\] 
where $(c,\bar X)\sim \nu_\GFF$. Note that  $X$ should not be thought of as a proper random variable as $\nu_\GFF$ has infinite mass. Yet, we find this viewpoint useful (this is also the point of view in \cite{Cargese}) as it is  intimately related to the construction of the so-called $\phi^4_2$ field. Indeed the latter field on the sphere $\S^2$/torus $\T^2$ corresponds to the probability measure on $\calH^{-\alpha}$
with the following explicit Radon-Nikodym derivative w.r.t $\nu_\GFF$
\[
d\nu_{\phi^4_2}(X) := \frac 1 {Z_\lambda} \exp(- \lambda \int_{\S^2} \Wick{X^4} )  d\nu_\GFF(X)\,.
\]
With this convention we deviate slightly from the notations in \cite{sphere} which uses instead  the writing $\bar X = c + X$.
\end{definition}

Inspired by LQFT, David, Kupiainen, Rhodes and Vargas rigorously construct in \cite{sphere} the following \textbf{grand canonical Liouville measure}.

\begin{theorem}[Grand canonical Liouville measure, \cite{sphere}]\label{th.LM}
$ $

\ni
Fix some real-parameter $\mu>0$, called the cosmological constant, some $\gamma\geq 0$ and some triple of distinct points
\footnote{If one wants to study the correlation functions of Liouville-QFT, one should consider instead $n\geq 3$ punctures $\{x_i\}_{1\leq i \leq n}$ and the analysis is the same. See \cite{sphere}.}
$\{x_1,x_2,x_3\}\subset \S^2$. Fix also any choice of coupling constants $(\alpha_1,\alpha_2,\alpha_3)$ which satisfy the so-called \textbf{Seiberg bounds} 
\bnum
\item (First Seiberg bounds)
$\forall i, \alpha_i < Q$
\item (Second Seiberg bounds)
$\sum_{i=1}^3 \alpha_i > 2Q$\,, 
\enum
where $Q$ was defined above. Then, if $\gamma<\gamma_c=2$, the following formal Radon-Nikodym derivative defines a non-trivial probability measure 
$\P^{(x_i,\alpha_i)}_\mu$ on $\calH^{-\alpha}(\S^2)$ for all $\alpha>0$:
\begin{align*}\label{}
\frac{d\P_{\mu}^{(x_i,\alpha_i)}}
{d\nu_\GFF} (X):= \frac 1 {Z^{(x_i,\alpha_i)}_\mu}
\exp\left(\sum_{i=1}^3 \alpha_i X(x_i) - \frac Q {2 \pi} \int_{\S^2} X(z) dz -  \mu \int_{\S^2} e^{\gamma X(z)} dz
\right) 
\end{align*}
\end{theorem}
It will be easier to work with this probability measure after the 
Cameron-Martin shift
\begin{align*}\label{}
X(z):= \sum_i \alpha_i G(x_i,z) + \hat X\,,
\end{align*}
where each $G(x_i, \cdot)$ denote the Green function for the Laplacian on the sphere $-\Delta G(z,\cdot) = 2\pi(\delta_z - \frac 1 {4\pi})$. (See~\eqref{e.GreenSphere} for an explicit formula on $\S^2$ as well as \cite{sphere, LecturesLQG} for other expressions). 
This Girsanov transform has the effect of removing the singular Dirac point masses at $\{x_i\}$ and leads to the following formal Radon-Nikodym derivative:

\begin{corollary}[\cite{sphere}]
$ $

\ni
Using the above Girsanov transform $X(z)= \sum_i \alpha_i G(x_i,z) + \hat X$, the Liouville measure on the sphere now reads:
\begin{align*}\label{}
\frac{d\P_{\mu}^{(x_i,\alpha_i)}}
{d\nu_\GFF} (\hat X):= \frac 1 {\hat Z^{(x_i,\alpha_i)}_\mu}
\exp\left(  \frac{\sum_i \alpha_i - 2Q} {4 \pi} \int_{\S^2} \hat X(z) dz -  \mu \int_{\S^2} e^{\sum_i \gamma \alpha_i G(x_i,z)} e^{\gamma \hat X(z)} dz
\right) 
\end{align*}
\end{corollary}

This construction theorem is proved in \cite{sphere}.  (The case of $n$-tuples $(x_i,\alpha_i)_{1\leq i \leq n}$ with $n\geq 3$ is also analysed there with the motivation of computing correlation functions between $e^{\alpha_i X(x_i)}$. See in particular the recent \cite{DOZZ} for the rigorous derivation of the so-called DOZZ formula for the  three point function). To make sense of the above formal Radon-Nikodym derivative, $\eps$-regularisations of the field are introduced in \cite{sphere}. For each small $\eps>0$, they consider $\hat X_\eps(x)$ to be $\<{\hat X,\eta_{x,\eps}}$ where $\eta_{x,\eps}$ is the uniform measure on $\p B_\eps(x)$ where $B_\eps(x)$ is the ball of radius $\eps$ around $x$ in $\S^2$ for the intrinsic round metric on $\S^2$. 
(which is of constant curvature $R=+2$). With this choice of regularisation, some appropriate renormalisation as $\eps\to 0$ is needed: instead of the ill-defined $e^{\gamma \hat X(z) }$, it is the Wick ordering $\Wick{e^{\gamma \hat X}} := \lim_{\eps\to 0} e^{\gamma \hat X_\eps - \frac {\gamma^2} 2 \Eb{\hat X_\eps^2}}$ which is  used.  See Remark \ref{r.Wick} below.
We refer to \cite{sphere} for more details (keeping in mind as mentioned above that the notations there are slightly different). 

\medskip

This grand-canonical probability measure has a very nice interpretation in terms of \textbf{planar maps} and a precise embedding conjecture can be written down in this grand canonical setting. See for example the enlightening discussions on this topic in \cite{LecturesLQG,sphere,torus}. 


\medskip

If one prefers instead to stick to planar maps of fixed size $n$ (\textbf{micro-canonical ensemble}), a unit-volume Liouville measure is also constructed in \cite{sphere} by in some sense conditioning $\P^{(x_i,\alpha_i)}_\mu$ to be of unit volume. This conditioning has the effect of factorising out the cosmological constant $\mu>0$ and leads to an explicit Radon-Nikodym derivative which has the following form
\begin{align*}\label{}
d\FK{\mu, 1}{(x_i,\alpha_i)}{\hat X}:= \frac 1 {Z^{(x_i,\alpha_i)}_{\mu, 1}}
\left[\int_{\S^2}  e^{\gamma  (\hat X(z) + \sum_{i=1}^3 \alpha_i G(x_i,z) )} dz
\right]^{-\frac {\sum_i \alpha_i - 2Q}{\gamma}} d\nu_\GFF(\hat X)\,.
\end{align*}
As we shall not use further this unit-volume Liouville measure, we refer to \cite{sphere, LecturesLQG}.
We now turn to the case of the torus which was settled in \cite{torus}. 

\subsection{Liouville measure on the torus}
We follow here \cite{torus} which constructs the Liouville measure on $\T^2$. As explained earlier, the Liouville measure on $\T^2$ will depend on the choice of a conformal structure which is parametrised by $\tau \in \H$ (in fact a sub-domain of $\H$). One way to picture the different possible choices of conformal structure is to work whatever $\tau\in \H$ is, on the same space $\T^2 := \C / (\Z + i\cdot \Z)$ equipped with the flat Riemannian metric 
\[
\hat g_\tau(x) dx^2 = |dx_1 + \tau dx_2|^2 \,\, \forall x=(x_1,x_2) \in \T^2
\]
The following probability measure on $\calH^{-\alpha}(\T^2)$ is built in \cite{torus} (we state the result after a similar Girsanov transform as on the sphere). 
\begin{theorem}[\cite{torus}]\label{th.LMt}
Fix some $\tau\in \H$ representing the choice of conformal structure on $\T^2$. 
Fix also some real-parameter $\mu>0$ (the cosm. constant), some $\gamma\geq 0$, some point $x_1\in \T^2$ and some  coupling constant $\alpha_1$ which satisfies the \textbf{Seiberg bounds} on the torus, namely
\[
0< \alpha_1 < Q
\]
Then, if $\gamma<\gamma_c=2$,  the following formal Radon-Nikodym derivative defines a non-trivial probability measure on $\calH^{-\alpha}(\T^2)$ for all $\alpha>0$: 
\begin{align*}\label{}
\frac{d\P_{\tau, \mu}^{(x_1,\alpha_1)}}
{d\nu_\GFF} (\hat X):= \frac 1 {\hat Z^{(x_1,\alpha_1)}_{\tau, \mu}}
\exp\left(  \frac {\alpha_1} {\lambda_{\hat g_\tau} (\T^2)}  \int_{\T^2} \hat X d\lambda_{\hat g_\tau} -  \mu \int_{\T^2} e^{\gamma \alpha_1 G_{\tau}(x_1,z)} e^{\gamma \hat X(z)} d\lambda_{\hat g_\tau}(z)
\right) 
\end{align*}
where $G_\tau(x_1,\cdot)$ is the $\tau$-Green function on $\T^2$ and $d\lambda_{\hat g_\tau}(z)$ is the volume form associated to $\hat g_\tau$. (See Section 3 in \cite{torus}).  (N.B. if one is willing to get back to the field $X$, an additional constant term $\frac Q 2 \log \mathrm{Im}(\tau)$ must be added, see \cite{torus}). 
\end{theorem}

\subsection{The SPDE under investigation}
$ $

\ni
Our goal in this work is to construct a stochastic evolution whose invariant measure is given by the grand-canonical Liouville measures $\P^{(x_i,\alpha_i)}_\mu$ and $\P_{\tau, \mu}^{(x_1,\alpha_1)}$ introduced above in Theorems \ref{th.LM} and \ref{th.LMt}. We are inspired here by the following two classical cases: 
\bnum
\item \textbf{$\Phi^4_d$ model.} On the $d$-dimensional torus $\T^d$, this corresponds to the informally written measure
\[
d\nu_{\phi^4_d}(X) := \frac 1 {Z_\lambda} \exp(- \lambda \int_{\S^d} X^4(z) dz )  d\nu_\GFF(X)\,.
\]
This model has a long and rich history in the field of {\em Constructive Field Theory}. 
It is already rather non-trivial to define it rigorously in dimension $d=2$ (\cite{Nelson}) and even much more involved in dimension $d=3$. Motivated by this celebrated model, Da Prato and Debussche  studied in their very influential paper \cite{dPD} the following SPDE which is aimed at preserving $\Phi^4_{d=2}$:
\begin{align}\label{e.DP}
\p_t \Phi = \frac 1 2 \Delta \Phi + C \Phi - \Phi^3 + \xi
\tag{$\Phi^4_d$}
\end{align}
where $\xi$ is a space-time white noise, for example on $\R_+ \times \T^2$ if we are interested in the $\Phi^4_2$ model on the two-dimensional torus $\T^2$. In \cite{HairerReg}, Hairer developped his theory of {\em Regularity Structures} which enabled him to solve (as a function of $\xi$) many highly non-trivial SPDEs including \textbf{dynamical $\Phi^4_3$ in $d=3$.} See also \cite{Gubi, Kup} for other approaches based respectively on {\em para-products} and {\em Renormalization group analysis}.  
\medskip

\item \textbf{Sine-Gordon model.} On the $d=2$ dimensional torus $\T^2$, Hairer and Shen study in \cite{sine} the following SPDE
\begin{align}\label{e.SG}
\p_t u = \frac 1 2 \Delta u + \sin(\beta u) + \xi\,
\end{align}
which is naturally associated to the so-called {\em Sine-Gordon model} in Constructive Field Theory. It corresponds to the following measure informally written as
\[
d\nu_{\phi^4_d}(X) := \frac 1 {Z_\lambda} \exp(- \lambda \int_{\T^2} \cos(\beta X(z)) dz )  d\nu_\GFF(X)\,.
\]
Using the Da Prato-Debussche framework, a local existence result for the SPDE~\eqref{e.SG} is proved in \cite{sine} for the range $\beta^2\in [0,4\pi)$. Using the regularity structures from \cite{HairerReg}, they manage to push the local existence further to $\beta^2\in[4\pi, \frac {16} 3 \pi)$. (The limiting value is conjectured to be $\beta^2=8\pi$ above which the SPDE becomes super-critical).  The paper \cite{sine} will be a constant source of inspiration in this text due to the important similarities between the Sine-Gordon SPDE~\eqref{e.SG} and the one we shall focus on. Indeed we will replace the non-linearity $e^{i \beta X}$ by the non-linearity $e^{\gamma X}$.   
\enum

\ni
As in the two celebrated cases above, since the Liouville measure (Theorems \ref{th.LM} and \ref{th.LMt}) is also of Hamiltonian form, we are led to consider the following formal SPDEs (For notational ease, we will drop the hat in $\hat X$ from then on).
\bi
\item On the torus $\T^2$
\begin{align}\label{e.DLT}
\p_t X = \frac 1 {4\pi}  \Delta X - \frac 1 2 \mu \gamma  e^{\alpha_1 \gamma G_\tau(x_1,\cdot)} e^{\gamma X} +   \frac {\alpha_1} { 2\lambda_{\hat g_\tau} (\T^2)}  + \xi 
\end{align}
\item On the sphere $\S^2$ 
\begin{align}\label{e.DLS}
\p_t X = \frac 1 {4 \pi} \Delta_{\S^2}  X - \frac 1 2 \mu \gamma \prod_i e^{\alpha_i \gamma G(x_i,\cdot)} e^{\gamma X} + \frac {\sum \alpha_i - 2Q} {8\pi}  + \xi 
\end{align}
where $\Delta_{\S^2}$ is the Laplace-Beltrami operator on $\S^2$ and $\xi$ is a space-time Gaussian white noise on $\R_+ \times \S^2$. 
\ei

\begin{remark}\label{}
Note that prior to the use of a suitable Girsanov transform (see Theorem \ref{th.LM}), the corresponding SPDE, say on the sphere would be of the following form  
\begin{align*}\label{}
\p_t X = \frac 1 {4 \pi} \Delta X - \frac 1 2 \mu \gamma e^{\gamma X} + \frac 1 2 \sum_{i=1}^3 \alpha_i \delta_{x_i}(\cdot)   - \frac Q {4\pi}  + \xi 
\end{align*}
which is much more singular. 
\end{remark}

Note that the SPDEs~\eqref{e.DLT}~\eqref{e.DLS} have a very asymmetric nature : the term $- e^{\gamma X(z)}$ prevents the solution $X(t,z)$ to reach high positive values while the global positive drift term $\frac{ \alpha_1}   {2\lambda_{\hat g_\tau}}$ (resp. $+\frac{\sum \alpha_i - 2 Q}{8\pi}$) prevents the field from diverging to $-\infty$. This is by the way an interesting way of interpreting the second Seiberg bound in Theorem \ref{th.LM}. This asymmetry is very different from the case of dynamical $\Phi^4_d$, where the vector field reads as follows: $\p_t \Phi = \Delta \Phi  - \Phi^3 + C \Phi + \xi$ and is thus symmetric. 
These considerations would be crucial to handle global (in time) existence results but are not essential for local existence results.
Indeed as pointed out in \cite{HairerReg}, the tools developped in this paper also provide a local existence result for the SPDE $\p_t \Phi = \Delta \Phi  \red{+} \Phi^3 + \xi$ which clearly has no chance to have global solutions. Because of this and also because the main difficulty for us will be to handle the exponential term $e^{\gamma X}$, we will focus in most of this work on the following simplified SPDEs (resp. on $\T^2$ and $\S^2$) 
\begin{align}\label{e.DLSs}
\p_t X  = \frac 1 {4\pi} \Delta X - e^{\gamma X} + \xi
\end{align}
This will have the advantage to make the proofs more readable. Another possible motivation for this simplification is that the same analysis applies to the more symmetric SPDE 
\begin{align}\label{e.DLSsinh}
\p_t X  = \frac 1 {4\pi} \Delta X - \sinh(\gamma X) + \xi\,,
\end{align}
which corresponds to the so-called $\cosh$-interaction in Quantum Field Theory. 
We will get back to the original SPDEs \eqref{e.DLT}, \eqref{e.DLS} which was motivated by LQFT only in Section \ref{s.Ext}.  See Theorem \ref{th.mainP} below. Finally, as mentioned above, yet another reason for this simplified equation is its close similarity with the dynamical Sine-Gordon SPDE~\eqref{e.SG} studied in \cite{sine} \footnote{Note the different choice of renormalisation for Sine-Gordon. See Remark \ref{r.shift}}

\subsection{Intermittency and SPDEs}\label{ss.intermit}

One interesting aspect of the above SPDE~\eqref{e.DLSs} is that it presents new features which are due to the {\em intermittent nature} of the non-linearity $e^{\gamma X}$. 
We refer for example to \cite{intermit} for a discussion of intermittency. In the present context, one may give the following two informal definitions of intermittency in the framework of the Besov spaces $\Cs^\alpha$ defined later in Subsection \ref{ss.Besov}. 
\bnum
\item Besov multi-fractality.  Given a distribution $f$, for any space-time point $z=(t,x)$, one may define its local Besov regularity $\alpha(z)$ similarly as one may define the local Hölder regularity of a continuous function. (Using say the setup introduced in Subsection \ref{ss.Besov} and relying on $\liminf_{\lambda \to 0}$). A distribution $f$ will be Besov mono-fractal if that local regularity happens to be  the same in each space-time point $z$. 
\item Besov intermittency. For random signals $f$, a related but different way to detect intermittency (often used in the study of turbulence) is through the behaviour of its moments. A time-space homogeneous random distribution $f$ will be intermittent if moments such as $\Eb{|\<{f,1_{B(z,\lambda)}}|^p}$ behave like $\lambda^{\xi(p)}$ when $\lambda\to 0$ for some \textbf{non-linear} spectrum $p\mapsto \xi(p)$. 
\enum

To our knowledge, all SPDEs which have been studied so far are mono-fractal for both of these definitions. Below are some examples.

\bnum
\item \textbf{KPZ.} The non-linearity $(\p_x h)^2$ is clearly a mono-fractal object. Indeed, it is the ``square of a white-noise'', i.e. not a well-defined distribution, but its regularizations are mono-fractal. (The local regularity may deviate from the average one only by $\log1/\eps$ factors). 
\item $\mathbf{\Phi^4_2}$ and $\mathbf{\Phi^4_3}$ are mono-fractal as well. Indeed, their local regularity in each space-time point is the same as for the linear stochastic heat equation. Another way to see the mono-fractal behaviour is that in some sense, the $p^{th}$ moments of the non-linearity behave linearly in $p$. 
\item \textbf{Sine-Gordon} $\mathbf{e^{i\, \beta \Phi}}$ is mono-fractal too. 
\item \textbf{Generalized KPZ.} (\cite{Bruned}) Non-linearities such as $g(u) (\p_x u)^2$ are mono-fractal. 
\enum

\smallskip
\ni
{\bf Misleading sub-criticality.} We will see in this work that the SPDE~\eqref{e.DLSs} is intermittent for both definitions. See Proposition \ref{pr.mom} which shows that $p^{th}$ moments of the non-linearity are governed by a non-linear spectrum $\xi_\s(p)$. In fact it can easily be shown (a classical fact in static Gaussian Multiplicative Chaos theory) that $p^{th}$ moments cease to exist when $p>8/\gamma^2$. It is also a classical fact that such GMC measures have a whole multi-fractal spectrum of regularities (corresponding to the so-called {\em thick points} of the field). See Remark \ref{r.sharp}.

One consequence of this intermittent behaviour is that it makes it harder to guess when the SPDE~\eqref{e.DLSs} becomes critical. Let us consider the following three cases to illustrate what we mean. 
\bnum
\item[a)] In the case of KPZ (resp. $\Phi^4_d$), if we zoom around a fixed space-time point using the change of variable $\hat h:= \lambda^{\frac d 2-1}h(\lambda^2 t, \lambda x)$ (resp. $\hat \Phi:= \lambda^{\frac d 2 -1} \Phi(\lambda^2 t, \lambda x)$) and let $\lambda\searrow 0$, then we obtain the same SPDE but where the non-linearity has shrinked by $\lambda^{1-\frac d 2}$ (resp. $\lambda^{4-d}$). This rescaling argument explains why $d_c=2$ for KPZ and $d_c=4$ for $\Phi^4_d$. 
\item[b)] In the case of Sine-Gordon (\cite{sine}), $\p_t u = \frac 1 2 \Delta u + \sin(\beta u) + \xi$, it is natural to play with the same rescaling $\hat u:= u(\lambda^2 t, \lambda x)$. In this case, the SPDE does not transform as simply as in the above two examples (the non-linearity is not polynomial), yet it is easy to check that {\em in law} (and assuming the solution locally looks like the solution of the SHE), the non-linearity is shrinking by $\lambda^{2-\frac{\beta^2} {4\pi}}$. This gives the correct prediction that $\beta_c=\sqrt{8 \pi}$ which corresponds to the Kosterlitz-Thouless phase transition. 
\item[c)] In the case of our SPDE $\p_t X = \frac 1 {4\pi} \Delta X - e^{\gamma X} + \xi$, one can try the same rescaling argument as for Sine-Gordon by zooming in with $\hat X:=X(\lambda^2 t, \lambda x)$. This time, again {\em in law} and assuming the solution locally looks like the solution of the SHE (which will be the case), the non-linearity is shrinking by $\lambda^{2\red{+} \g2}$. At first sight, this suggests this SPDE will be sub-critical for all parameters $\gamma >0$! The solution to this paradox is precisely the intermittent nature of the non-linearity. This zooming argument is correct but is applied at a typical space-time point according to Lebesgue measure. At those points, the non-linearity is indeed ``smaller and smaller'' as $\gamma$ increases. Instead, the rescaling argument should be applied around atypical space-time points for Lebesgue which are more meaningful for the non-linearity (i.e. around the appropriate choice of {\em thick points} of the field $X$). It is therefore an interesting problem in this intermittent case to detect at which value $\gamma=\gamma_c$ the SPDE~\eqref{e.DLSs} will stop being critical. (See Question \ref{q.c} and the discussion before). 
\enum

\subsection{Main results}\label{ss.results}
$ $

\ni
Let us first settle some notations. Recall $\xi$ denotes a space/time white noise on either $\R\times \T^2$ or $\R \times \S^2$ (i.e. the Gaussian process with covariance $\Eb{\xi(s,x) \xi(t,y)}= \delta_{(s,x), (t,y)}$). \margin{ALREAYDY mentioned before!?}
We will need to regularise our singular distributions such as $\xi$ by working with some smoothened noise $\xi_\eps$. 
On the flat case $\R\times \T^2$, as in \cite{sine} say, we fix some smooth and compactly supported function $\varrho: \R \times \R^2\to \R_+$ integrating to 1. We then consider $\xi_\eps:= \varrho_\eps * \xi$, with $\varrho_\eps(t,x)=\eps^{-4} \varrho(\eps^{-2} t, \eps^{-1} x)$. 
In the curved case of the sphere, see Subsection \ref{ss.setupS} for our precise setup.  
See also \cite{sphere, LecturesLQG} for a slightly different smoothing on the sphere.


\begin{theorem}\label{th.mainT}
Let $0\leq \gamma <\gamma_{dPD}:=2\sqrt{2}-\sqrt{6}\approx 0.38$. 
\margin{or $0< \gamma^2 < \gamma_1^2 = 14-8\sqrt{3}$} 
Fix an initial condition of the form $u(0) = \Phi(0)+ w$ (where $\Phi$ is the log-correlated field induced by the linear heat equation, see Subsection \ref{ss.dPD} and with $w\in \calC^{a}(\T^2)$ for some $a>0$). Consider the solution $X_\eps$ to 
\margin{C'est ICI que la constante $C_\rho$ est importante a prendre en consideration dans la definition de $\Wick{e^{\gamma X_\eps}}$ !}
\begin{align*}\label{}
\begin{cases}
& \p_t X_\eps = \frac 1 {4 \pi} \Delta X_\eps - C_\varrho \eps^{\frac {\gamma^2} 2} e^{\gamma X_\eps} + \xi_\eps  \\
& X_\eps(0,\cdot) = \Phi_\eps(0) + w
\end{cases}
\end{align*}
There is a way to tune the constant $C_\varrho$ (depending only on $\gamma$ and the mollifier  $\varrho$) as well as an a.s. positive time $T>0$ such that $X_\eps$ converges in probability in the space $\Cs^{-1}([0,T]\times \T^2)$ to a limiting distributional process $X$ which is independent of the mollifier $\varrho$. (See Subsection \ref{ss.Besov} for the definition of Besov spaces $\calC_\s^\alpha$). 
\margin{C: I almost copied and paste here... Make some changes}
\end{theorem}

\begin{remark}\label{r.Wick}
Note that the non-linearity $e^{\gamma X}$ has been renormalized in order to have a limiting solution for the SPDE. The appropriate way of renormalizing here (as for dynamical $\Phi^4_2$ for example) is the so-called \textbf{Wick ordering}. For the linear heat equation, i.e.  the Gaussian process $\Phi$ introduced in subsection \ref{ss.dPD}, this corresponds to 
\begin{align}\label{e.Wick}
\Wick{e^{ \gamma \Phi_\eps}}:= e^{-\g2 \Eb{\Phi_\eps(0)^2}} e^{\gamma \Phi_\eps} \sim_{\eps \to 0} C_\varrho \eps^{\g2} e^{\gamma \Phi_\eps}
\end{align}
where the asymptotics as $\eps \to 0$ is given by Proposition \ref{l.key} and where the constant $C_\varrho$ is set to be equal to $e^{-\g2 \hat C_\varrho}$ (the latter constant $\hat C_\varrho$ is defined in Proposition \ref{l.key}). This explains the above renormalisation as $\eps \to 0$. 
\end{remark}

\begin{theorem}[* Assuming a spherical Schauder estimate, see Remark \ref{r.SchauderSphere} *]\label{th.mainS} $ $

Let $0\leq \gamma < \gamma_{dPD} = 2\sqrt{2} - \sqrt{6}$. Fix an initial condition of the form $u(0) = \Phi(0)+ w$ (with $w\in \calC^{a}(\S^2)$ for some $a>0$). Consider the solution $X_\eps$ to 
\begin{align*}\label{}
\begin{cases}
& \p_t X_\eps = \frac 1 {4 \pi} \Delta X_\eps - C_\varrho \eps^{\frac {\gamma^2} 2} e^{\gamma X_\eps} + \xi_\eps  \\
& X_\eps(0,\cdot) = \Phi_\eps(0) + w
\end{cases}
\end{align*}
There is a way to tune the constant $C_\varrho$ (depending only on $\gamma$ and the mollifier  $\varrho$)
as well as an a.s. positive time $T>0$  such that $X_\eps$ converges in probability in the space $\Cs^{-1}([0,T]\times \S^2)$ to a limiting distributional process $X$ which is independent of the mollifier $\varrho$. (See Subsection \ref{ss.setupS} for the definition of these spaces in the case of the sphere). 
\margin{C: I almost copied and paste here... Make some changes}
\end{theorem}

\begin{remark}\label{}
As in \cite[Section 2.]{sine} or \cite{Labbe}, one could certainly strengthen our hypothesis on the initial condition $X(0,\cdot)$ but we decided to limit ourself to the simplest setting as in \cite{dPD}. \margin{See the discussion in Section \ref{s.??}. Probably no such discussion} 
Note that our hypothesis is very meaningful as the Liouville measure introduced in Theorem \ref{th.LM} is absolutely continuous w.r.t $\Phi(0)$ and is of the above form so that our hypothesis allows us to start from equilibrium. 
\end{remark}

As mentioned above, these results describe the Da Prato-Debussche phase and do not  rely on the recent breakthrough theories for SPDEs, i.e. regularity structures,  paraproducts or renormalisation techniques.
Interestingly, using the fact that the distribution valued process $\Wick{e^{\gamma \Phi}}$ is not a generic highly oscillatory distribution but rather a (singular) positive measure, we manage to push the Da Prato-Debussche limit further by exploiting the positivity. We only obtain this way a strong solution to the SPDE in this regime (measurable w.r.t the input $\xi$) and do not control the convergence of $X_\eps \to X$ in this case. 

\begin{theorem}\label{th.main2}
When $\gamma_{\mathrm{dPD}} \leq \gamma < \gamma_{\mathrm{pos}}=2\sqrt{2} -2\approx 0.83$, we can still define a strong  solution of the SPDE \ref{e.DLSs} through a fixed point problem determined by the driving noise $\xi$. We do not have access to convergence $X_\eps \to X$ with our method in this regime.
\end{theorem}

This will be proved in Section \ref{s.POS}. It can be argued that this threshold $\gamma_{\mathrm{pos}}$ corresponds to the critical parameter $\beta_c=\sqrt{8\pi}$ for Sine-Gordon. See the discussion in Section \ref{s.RS}. (It seems in particular that regularity structures used at arbitrary large order should not push local existence of solutions for values of $\gamma$ larger than $\gamma_{\mathrm{pos}}$). 

\smallskip

Finally,  we will get back to the original SPDEs \ref{e.DLT}, \ref{e.DLS} which was motivated by LQFT in Section \ref{s.Ext}. We will see that the $\alpha_i$-logarithmic singularities which arise near the punctures $\{x_i\}$ do have a significant impact on the underlying regularity. In particular, if the \textbf{coupling constants} $(\alpha_i)$ are chosen too large (yet satisfying the Seiberg bound $\alpha_i < Q$), then the regularity drops too low in order to settle a fixed point argument. Because of this, the Da Prato threshold $\gamma_{dPD}$ and the above $\gamma_{\mathrm{pos}}$ threshold need to be adapted accordingly. We will prove the following result below. 
See also Corollary \ref{C.CoupConst} for an explicit computation of these modified thresholds in the most relevant case (for LQG and embedding of planar maps) of $\alpha_i:=\gamma$. 

\begin{theorem}\label{th.mainP}
Let $\gamma< 2$ and $\{(\alpha_i,x_i)\}_i$ be punctures whose coupling constants $\alpha_i$  satisfy the Seiberg bounds. 
The main results above for the simplified SPDE~\eqref{e.DLSs} (i.e. Theorem \ref{th.mainT}, \ref{th.mainS} and Theorem \ref{th.main2}) also hold for the Liouville SPDEs  \eqref{e.DLT}, \eqref{e.DLS}:
\[
\begin{cases}
& \p_t X = \frac 1 {4\pi}  \Delta X - \frac 1 2 \mu \gamma  e^{\alpha_1 \gamma G_\tau(x_1,\cdot)} e^{\gamma X} +   \frac {\alpha_1} { 2\lambda_{\hat g_\tau} (\T^2)}  + \xi  \\
& X(0,\cdot) = \Phi(0) + w\,.
\end{cases}
\]
except the thresholds $\gamma_{dPD}$ and $\gamma_{\mathrm{pos}}$ in these theorems have to be modified as follows. Let us first introduce some notations. On the sphere $\S^2$, if $\{(x_i,\alpha_i)_i\}$ are the coupling parameters at the punctures (which satisfy Seiberg bounds), let $\hat \alpha:= \max_i \alpha_i$. On the torus, let $\hat \alpha:= \alpha_1$. Then, the following holds. 
\bnum
\item If $\gamma$ is chosen small enough so that 
\[
R(\hat \alpha,\gamma):= \bigl[(\frac \gamma {2\sqrt{2}} - \hat \alpha \gamma)\wedge 0 \bigr] + \g2 - 2\sqrt{2} \gamma > -1\,,
\]
then local existence as well as convergence as $\eps\to0$ hold as in Theorem \ref{th.mainT}. 
\item If $R(\hat \alpha, \gamma)>-2$, then we obtain with our methods only the local existence as in Theorem \ref{th.main2}. 
\margin{...which preserves $\P_{\tau=i,\mu}^{x_1,\alpha_1}$ HAS DROPPED, maybe this is too bad ?}
\enum 

%

\end{theorem}

\paragraph{\bf Acknowledgments}

I wish to thank the three anonymous referees for many helpful comments which helped improving the paper.
I wish to thank Nathanael Berestycki, Charles-Edouard Bréhier, Louis Dupaigne, Francesco Fanelli, Ivan Gentil, Martin Hairer, Dragos Iftimie, Antti Kupiainen, Claudio Landim, Rémi Rhodes, Hao Shen, Fabio Toninelli,  Nikolay Tzvetkov, Vincent Vargas, Julien Vovelle for  inspiring discussions  or important comments on the manuscript and  Jean-Christophe Mourrat for very stimulating discussions at several stages of this work.  The author is partially supported by the ANR grant Liouville ANR-15-CE40-0013 and the ERC grant 676999-LiKo. 

\section{The case of the torus $\T^2$}\label{s.TORUS}

In this section, we will restrict ourselves to the case of the two-dimensional torus $\T^2$.

\subsection{Da Prato-Debussche approach}\label{ss.dPD}

It is known since the breakthrough paper \cite{dPD} (see also the recent surveys on the topic \cite{Hendrik, ICM, Braz}) that in order to give a meaning (after suitable renormlisation) to such SPDEs, it is very convenient to make a first-order expansion around the solution of the linearised equation $\p_t \Phi = \frac 1 {4\pi} \Delta \Phi + \xi$ (the stochastic heat equation) as follows:
\begin{align*}\label{}
X(t,x) := \Phi(t,x) + v(t,x)\,.
\end{align*}
Note that higher order such expansions are at the root of the theory of {\em Regularity structures} in \cite{HairerReg}. By a simple time-change, one may as well consider the following SPDE which has the advantage to be related to the {\em standard} heat kernel. 
\begin{align*}\label{}
\p_t \Phi = \frac 1 {2} \Delta \Phi + \sqrt{2\pi}\xi
\end{align*}
By {\em Duhamel's principle}, the solution of this heat-equation
is given by $\Phi = \sqrt{2\pi} \tilde K* \xi$ where $\tilde K$ is the heat-kernel on the torus.  
As explained for example in \cite{sine}, it turns out to be more convenient to expand at first order around an approximate solution of the heat-equation given by 
\begin{align}\label{e.K}
\Phi:= \sqrt{2\pi} \, K * \xi
\end{align}
where $K$ is a compactly supported (in space-time) kernel which coincides with the heat-kernel $p_t^{\T^2}(x,y)$ in the neighbourhood of $t=0$. This technical step will be even more important once on the sphere $\S^2$ in Section \ref{s.sphere} as it will avoid us dealing with subtle {\em cut-locus} issues. See Proposition \ref{pr.nag} and below. With this definition of $\Phi$, one easily check (see \cite{sine}) the existence of a smooth correction $R(t,z)$ such that 
\begin{align*}\label{}
\p_t \Phi = \frac 1 2 \Delta \Phi + R + \sqrt{2\pi}\xi
\end{align*}
\margin{In sine paper: p20,  we will use the symbol R to denote a generic such function which can possibly change from one line to the next.}

Following \cite{dPD}, if one formally plugs $X=\Phi + v$ into our simplified SPDE $\p_t X = \frac 1 2 \Delta X - e^{\gamma X} + \sqrt{2\pi} \xi$, we end up with the equation 
\begin{align}\label{e.dPD}
\begin{cases}
& \p_t v = \frac 1 2 \Delta v  - e^{\gamma \Phi} e^{\gamma v} - R \\
& v(t=0,x)= X(t=0,x)-\Phi(t=0,x)
\end{cases}
\end{align}
As in \cite{dPD} and works since then, we will solve this SPDE via a fixed point argument in an appropriate Banach space.


\subsection{Space-time Besov spaces}\label{ss.Besov}
Let us briefly introduce the functional framework we shall use. We follow very closely here \cite{sine, Hendrik, ICM} to which we refer for a more detailed exposition.
Space-time points will be denoted by $z=(t,x)$ and we
will use throughout the following parabolic space-time distance on $\R \times \R^d$. 
\begin{align}\label{e.ds}
\|z\|_\s := \|(t,x)\|_\s = |t|^{1/2} + \| x\|_2
\end{align}
This distance induces a parabolic distance on $\R \times \T^d$. (See Subsection \ref{ss.setupS} below for the case of the sphere.)

Let us introduce some notation. For any integer $m\geq 1$, let $B_m$ be the space of smooth test functions $f : \R \times \T^d \to \R$ which are supported on the unit ball around $(0,0)$ for the parabolic distance $\|\cdot\|_\s$ and satisfy 
\[
\| f \|_{C^m} := \sup_{\beta, |\beta|_{\s}\leq m} \sup_{z\in \R \times \T^d} |D^\beta f(z)| \,\, \leq 1 \,.
\]
Following \cite{HairerReg, Hendrik, ICM}, we introduce the following \textbf{space-time Besov spaces} of negative regularity.
\begin{definition}\label{d.Balpha}
For any $\alpha<0$, let $\calC_\s^\alpha$ be the space of distributions $\eta$ on $\R\times \T^d$ s.t. for any $T>0$, 
\begin{align}\label{e.Balpha}
\|\eta \|_{\calC^\alpha_\s(T)}:= \sup_{z\in [-T,T]\times \T^d}\sup_{f \in B_m} \sup_{\lambda \in(0,1]} \left| 
\frac
{\<{\eta, S_z^\lambda f}}
{\lambda^\alpha}
\right|<\infty
\end{align}
where $m:=\ceil{-\alpha}$ and if $z=(t,x)$, 
\[
S_z^\lambda f(s,y) := \lambda^{-d-2} f(\lambda^{-2}(s-t), \lambda^{-1}(y-x))
\]
These spaces correspond to the more standard notation $B^\alpha_{\infty,\infty}$ with the additional fact that we are in space/time here. 
\end{definition}


For positive regularity $\beta\in(0,1)$, we will use the following classical definition of $\calC_\s^\beta$.

\begin{definition}\label{d.Bbeta}
For any positive regularity $\beta\in (0,1)$, let  $\calC_\s^\beta$ be the space of functions $f$ on $\R \times \T^2$, such that
\begin{align*}\label{}
\|f\|_{\calC_\s^\beta}& := \| f\|_{\infty} + \sup_{0<\|z-z'\|_\s\leq 1} \frac{|f(z)-f(z')|}{\|z-z'\|_\s^\beta} <\infty 
\end{align*}
There are several other equivalent ways to define $\calC_\s^\beta$. See for example Remark 2.16 in \cite{tightness}. See \cite{Hendrik, ICM} for equivalent definitions more in the spirit of~\eqref{e.Balpha}.
\end{definition}


Finally, for any fixed $t>0$ and $\alpha\in (-\infty,0) \cup (0,1)$, let  $\calC^{\alpha}_\s(\Lambda_t)$ be the Besov space introduced in Definitions \ref{d.Balpha}, \ref{d.Bbeta} on $\Lambda_t:=[0,t]\times \T^2$. 
\medskip

\subsection{Multiplying distributions}
The above functional setup is very convenient to multiply distributions which do not have pointwise values. The key result of this flavour is the following Theorem.   

\begin{theorem}[Thm 2.52 in \cite{Chemin} or Prop 4.11 in \cite{HairerReg}]\label{th.mult}
Suppose $\alpha+\beta>0$, then there exists a  bilinear form $B(\cdot, \cdot): \calC_\s^\alpha \times \calC_\s^\beta \to \calC_\s^{\alpha\wedge \beta}$ satisfying 
\bi
\item[i)] $B(f,g)$ coincides with the classical product when $f,g$ are smooth. 
\item[ii)] There exists $C>0$ s.t. for any $f\in \calC_\s^\alpha, g\in \calC_\s^\beta$, 
\[
\|B(f,g)\|_{\calC_\s^{\alpha \wedge \beta}} \leq C \| f \|_{\calC_\s^\alpha} \| g \|_{\calC_\s^\beta}
\]
\ei
N.B. The criterion $\alpha+\beta >0$ is optimal: such a bilinear map does not exist if $\alpha + \beta \leq 0$. 
\end{theorem}

This result shows that one can at least start Picard's iteration scheme if our rough object $e^{\gamma \Phi}$ is of high enough regularity.

\subsection{Parabolic Schauder's estimate}

The point of a parabolic Schauder's estimate is to quantify the fact that the solution $v(t,x)$ to a linear heat-equation
\begin{align}\label{e.DUH}
\begin{cases}
&\p_t v = \frac 1 2 \Delta v + f \\
& v(t=0, x)= v_0(x)
\end{cases}
\end{align}
is ``two units more regular'' than the driving space-time function $f$. There are many ways to quantify this phenomenon. We will measure this gain of regularity using the above space-time Besov spaces $\calC_\s^\alpha$. It will be convenient to decompose  Duhamel's principle into two operators $K$ and $P$. Recall that the solution to the heat equation~\eqref{e.DUH} is given by 
\[
v(s,x)= \int_0^s e^{\frac u 2 \Delta}f(s-u,x) du + e^{\frac s 2 \Delta} v_0
\]
which motivates the introduction of the operators $K$ and $P$ as follows. 
\begin{definition}\label{d.K}
For any fixed $t>0$ and any $f\in \Cs^\infty(\Lambda_t)$, let $K=K_{(t)}$ be the operator  
\begin{align*}\label{}
K(f)(s,x):= \int_0^s e^{\frac u 2 \Delta}f(s-u,x)du\,.
\end{align*}
\end{definition}

\begin{definition}\label{d.P}
For any $v_0$ a continuous function in $\calC(\T^d)$, let $P(v_0)$ be the space-time function 
\[
P(v_0) : (s,x) \mapsto e^{\frac s 2 \Delta} v_0\,.
\]
\end{definition}
The solution $(s,x) \mapsto v(s,x)$ to the heat-equation~\eqref{e.ds} can thus be written 
\begin{align*}\label{}
v = K(f) + P(v_0)
\end{align*}


We will rely on the following parabolic Schauder estimate.

\begin{proposition}[Parabolic Schauder estimate]\label{th.SchauderP}
For  any $\alpha<0$ and any small $\kappa>0$ s.t. 
$\alpha+2 -\kappa \in (0,1)$, there exists $C>0$ 
 such that for any $0<t<1$, 
\[
\| K(f) \|_{\Cs^{\alpha + 2 - \kappa}(\Lambda_t)} \leq C t^{\kappa/2} \| f \|_{\Cs^\alpha(\Lambda_t)}
\]
\end{proposition}
We did not find a proper reference for this estimate  which is implicitly used in several works such as \cite{sine} or \cite{Hendrik}. (The proof based on Littlewood-Paley in the version 1 of this paper had a mistake and it appears that Littlewood-Paley techniques are not convenient to handle the parabolic case). The best place in the literature is Section 14.3 on Schauder estimates in the book \cite{BookHairer} in which essentially contains a proof of this estimate. 

We shall also need the following simple Lemma
\begin{lemma}\label{l.simple}
Let $\alpha<0$ and $\kappa>0$ s.t. $\alpha+2-\kappa\in (0,1)$. There exists $C>0$ s.t. for any  initial condition $v_0$ in $\calC^{\alpha+2-\kappa}(\T^2)$
\[
\| P(v_0) \|_{\Cs^{\alpha+2-\kappa}(\Lambda_1)} \leq C \|v_0\|_{\calC^{\alpha+2-\kappa}(\Lambda_1)}
\]
\end{lemma}

\ni
{\em Proof:}
It follows from the fact that the heat kernel induces a contraction on Hölder spaces. Indeed, taking $\beta:=\alpha+2-\kappa\in (0,1)$, 
\begin{align*}\label{}
|P_t(f)(x) - P_t(f)(y)|
& = |\Eb{f(x+Z_t) - f(y+Z_t)}| \\
& \leq \|f\|_\beta \|x-y\|^\beta\,,
\end{align*}
(where $(Z_t)_{t\geq 0}$ is a Brownian motion on $\R^2$ and one thinks of $f$ as the periodic extension of $v_0$ to $\R^2$).
To prove the Lemma it is sufficient to extend this inequality to space/time points in $\Lambda_1 \times \T^d = [0,1]\times \T^d$. 
\begin{align*}\label{}
|P_t(f)(x) - P_s(f)(y)| & = 
|\Eb{f(x+Z_s + Z_{t-s}) - f(y+Z_s)}|\\
&  \leq \|f\|_\beta \Eb{\| x - y + Z_{t-s}\|^\beta} \\
& \leq \| f \|_\beta \Eb{\| x - y + Z_{t-s}\|^2}^{\beta/2} \text{   (by Jensen)} \\ 
& = \|f\|_\beta (\|x-y\|^2 + |t-s|)^{\beta/2}
\end{align*}
\qed

As explained in Remark \ref{r.Wick}, the non-linearity $e^{\gamma X}$ needs to be renormalized. Let us assume for the moment that we are able to construct an object $\Wick{e^{\gamma \Phi}}$ which satisfies 
\bi
\item[i)]  $\Wick{e^{\gamma \Phi}}$ belongs a.s. to $\calC_\s^{\alpha}$. 
\item[ii)] It is the a.s. limit in $\calC_\s^{\alpha}$ of its regularisations $\Wick{e^{\gamma \Phi_\eps}}$
\ei
(This will be proved in Theorem \ref{th.convergence}). 
We shall now explain how to set-up a fixed point argument sufficiently stable in its arguments so that it will provide a solution to~\eqref{e.dPD} (where $e^{\gamma \Phi}$ which is ill-defined is replaced by $\Wick{e^{\gamma \Phi}}$) which will be the limit of $\eps$-regularisations.


\subsection{Settling the fixed point argument}\label{ss.settle}
In what follows we fix $\alpha<0$ and $\kappa>0$ s.t. $\beta:=\alpha+2 -\kappa \in (0,1)$. We also assume that $\alpha+\beta>0$ in order to be able to use the multiplication theorem.

Recall we wish to solve~\eqref{e.dPD}. 
Again by {\em Duhamel's principle} and using the operators defined in Definitions \ref{d.K} and \ref{d.P}, we are looking for a fixed point of the map $F :  
\calC_\s^{\alpha+2 -\kappa} \to \calC_\s^{\alpha+2 -\kappa}$ defined by 
\begin{align}\label{}
F \,\, : \,\, u \mapsto \Bigl( (s,x)\mapsto K(-\Wick{e^{\gamma \Phi}} e^{\gamma u} - R)(s,x) + P(v_0)(s,x)\Bigr)
\end{align}
It will be important to keep track of the fixed parameters (inputs) which define this map $F$. Namely we introduce the following map. 
\begin{definition}\label{d.FPmap}
For any fixed choices of $t>0$, $\Theta\in \calC_\s^\alpha$, $v_0\in \calC^{\alpha+2-\kappa}(\T^2)$ and $R\in \calC_\s^\infty$, consider the following map on $\calC_\s^{\alpha+2 -\kappa}(\Lambda_t)$:
\begin{align}\label{}
F_{t,\Theta, v_0,R} \,\, : \,\, u \mapsto \Bigl( (s,x)\in \Lambda_t \mapsto K(- \Theta e^{\gamma u} - R)(s,x) + P(v_0)(s,x)\Bigr)
\end{align}
N.B. By the multiplication theorem \ref{th.mult}, this map is well defined if we suppose $2\alpha+ 2 -\kappa>0$. This is in some sense the barrier from Da Prato-Debussche approach that the non-linearity needs to live at least in $\calC_\s^{-1+\delta}$.  
\end{definition}

The key fact is that this map is a contraction on a suitably chosen subset of the Banach space $\calC_\s^{\alpha+2-\kappa}(\Lambda_t)$ if $t$ is chosen sufficiently small. More precisely we have:
\begin{proposition}\label{pr.FP}
For any $M>0$, there exists $t=t(M)>0$ and $r=r(M)>0$ so that for any choice of $(\Theta, R, v_0)$ satisfying 
\begin{align}\label{e.BB}
\|\Theta\|_{\calC_\s^\alpha(\Lambda_{t=1})} \vee \|R\|_{C_\s^{\alpha+2-\kappa}} \vee \|v_0\|_{\calC^{\alpha+2-\kappa}} < M 
\end{align}
Then, the map $F_{t,\Theta, v_0,R}$ is at least $1/2$-contracting from the ball  $B_{\calC_\s^{\alpha+2-\kappa}(\Lambda_t)}(0,r) \subset \calC_\s^{\alpha+2-\kappa}(\Lambda_t)$ to itself. 
This shows the existence of a unique fixed point $v_{t,\Theta,v_0,R}\in \calC_\s^{\alpha+2-\kappa}(\Lambda_t)$ for the map $F_{t,\Theta,v_0,R}$. Furthermore, the map
\[
(\Theta,v_0,R) \mapsto v_{t(M), \Theta, v_0, R}\in \calC_\s^{\alpha+2-\kappa}(\Lambda_{t(M)})
\]
is continuous on the open set defined by the constraint~\eqref{e.BB}. 
\end{proposition}

\ni
{\em Proof. (Sketch)}
We need to first check that for well chosen $t=t(M), r=r(M)$, the map $F$ preserves the ball of radius $r$ in $\calC_\s^{\alpha+2-\kappa}(\Lambda_t)$. For this note that 
\begin{align*}\label{}
& \| K(-\Theta e^{\gamma u} - R) + P(v_0)\|_{\calC_\s^{\alpha+2-\kappa}} \\
& \leq C t^{\kappa/2} \| \Theta e^{\gamma u} \|_{\calC_\s^{\alpha}} 
+ C t^{\kappa/2} \|R\|_{\calC_\s^{\alpha}} + \|P(v_0)\|_{\calC_\s^{\alpha+2-\kappa}} \\
& \leq C t^{\kappa/2} \|\Theta\|_{\calC_\s^\alpha} \|e^{\gamma u}\|_{\calC_\s^{\alpha+2-\kappa}}
+  C t^{\kappa/2} \|R\|_{\calC_\s^{\alpha+2-\kappa}} + C \|v_0\|_{\calC^{\alpha+2-\kappa}}
\end{align*}
where we used Proposition \ref{th.SchauderP} as well as Lemma \ref{l.simple}. We now need an estimate on $\|e^{\gamma u}\|_{\calC_\s^{\alpha+2-\kappa}}$ which is given to us by Lemma \ref{l.contract} below. By first choosing $r$ sufficiently large and then $t\leq 1$ sufficiently small we conclude that $F$ preserves the ball of radius $r$.

For the contraction property, when computing $\|F(u)- F(u')\|_{\calC_\s^{\alpha+2-\kappa}}$, the term coming from the initial condition $v_0$ as well as the term $R$ cancel out. We end up 
with controlling 
\begin{align*}\label{}
& \| K(\Theta e^{\gamma u}) -K(\Theta e^{\gamma u'}) \|_{\calC_\s^{\alpha+2-\kappa}} \\
& \leq C t^{\kappa/2} \| \Theta (e^{\gamma u} - e^{\gamma u'}) \|_{\calC_\s^{\alpha}} \\
& \leq C t^{\kappa/2} \|\Theta\|_{\calC_\s^\alpha} \|e^{\gamma u}-e^{\gamma u'}\|_{\calC_\s^{\alpha+2-\kappa}} \\
& \leq C t^{\kappa/2} \|\Theta\|_{\calC_\s^\alpha} C_r \, \gamma \|u-u'\|_{\calC_\s^{\alpha+2-\kappa}})\,,
\end{align*}
again by Lemma \ref{l.contract} below. By possibly choosing $t=t(M)$ even smaller than in the first step, we get the $1/2$-contraction property.

Finally, to prove the continuity statement, note that for any $u$ in the ball of radius $r$ of $\calC_\s^{\alpha+2-\kappa}(\Lambda_t)$ and any $(\Theta,v_0,R), (\Theta',v_0',R')$, one has 
\begin{align*}\label{}
& \|F_{t,\Theta,v_0,R}(u) - F_{t,\Theta',v_0',R'}(u)\|_{\calC_\s^{\alpha+2-\kappa}} \\
& \leq C t^{\kappa/2}  \|(\Theta -\Theta')(e^{\gamma u})\|_{\calC_\s^{\alpha}} + C t^{\kappa/2} \|R -R'\|_{\calC_\s^\alpha}  + \|P(v_0 -v_0')\|_{\calC_\s^{\alpha+2-\kappa}} \\
& \leq C t^{\kappa/2} \|(\Theta -\Theta')\|_{\calC_\s^\alpha} \|e^{\gamma u}\|_{\calC_\s^{2+\alpha-\kappa}} 
+ C t^{\kappa/2} \|R- R'\|_{\calC_\s^{\alpha+2-\kappa}}  + C \|(v_0 -v_0')\|_{\calC_\s^{\alpha+2-\kappa}}\,,
\end{align*}
again by Proposition \ref{th.SchauderP} and Lemma \ref{l.simple}. Now to show that both fixed points are close, it is enough to start Picard's iteration scheme for $F_{t,\Theta',v_0',R'}$ from the fixed point of the other $F_{t,\Theta,v_0,R}$ and use the fact that it is $1/2$-contractive.

To conclude the proof of the proposition, we are thus left with the following simple Lemma which is often implicit in the literature (for other non-linearities such a $u^3$ for the dynamical $\Phi^4_d$ model for example).

\begin{lemma}\label{l.contract}
Fix some regularity $\beta\in (0,1)$. For any $r>0$, there is $C=C_r>0$ s.t. for any $u, v \in B_{\calC_\s^\beta}(0,r)$, one has 
\begin{align*}\label{}
\| e^u - e^{v}\|_{\calC_\s^\beta} \leq C \| u - v\|_{\calC_\s^\beta}
\end{align*}
\end{lemma}
The proof is an elementary computation. Let us first deal with the $\|\cdot\|_\infty$ contribution in the definition of $\|\cdot\|_{\calC_\s^\beta}$ (recall Definition \ref{d.Bbeta}). 
For this part, note that $x\mapsto e^x$ is $c$-lipschitz on $[-r,r]$ for some $c=c_r$. In particular, for any space-time point $z$, $|e^u(z)- e^v(z)| \leq c |u(z)-v(z)|$. For the second part, one has for any $z\neq z'$, 
\begin{align*}\label{}
\frac
{|(e^{v} - e^{u})(z) -(e^{v} - e^{u})(z')|}
{\|z-z'\|_\s^\beta}
& = 
\frac{
|e^{u(z)}[e^{v-u}(z)-1] - e^{u(z')}[e^{v-u}(z')-1]|
}
{\|z-z'\|_\s^\beta} \\
&\hspace{- 3cm} = 
\frac{
|e^{u(z)}[e^{v-u}(z)-1] - e^{u(z)}[1+O(\|u\|_{\calC_\s^\beta} \|z-z'\|_\s^\beta)][e^{v-u}(z')-1]|
}
{\|z-z'\|_\s^\beta} \\
&\hspace{- 3cm} \leq 
\frac{
|e^{u(z)}[e^{v-u}(z)-e^{v-u}(z')]|}
{\|z-z'\|_\s^\beta} 
+O(\|u\|_{\calC_\s^\beta}) e^{u(z)} |e^{v-u}(z')-1]| \\
&\hspace{- 3cm} \leq 
C e^{\|u\|_\infty} \|v-u\|_{\calC_\s^\beta} 
+O(\|u\|_{\calC_\s^\beta}) C\, e^{\|u\|_\infty} \|v-u\|_\infty\,,
\end{align*}
where we relied several times on the fact that exponential is $C$-Lipschitz on $[-r,r]$. Also, from the first to second line, we used that $|e^{u(z')-u(z)}-e^0| \leq C_r\, |u(z)-u(z')| \leq C \|u\|_{\calC_\s^\beta} \|z'-z\|_\s^\beta$. \qed

\subsection{Proof of Theorem \ref{th.mainT}}

Let us briefly summarise: we wish to prove a \textbf{local existence} result for the simplified SPDE \ref{e.DLSs}. After a time-change and the Da Prato-Debussche change of variable $X = \Phi + v$ (where $\Phi:=K*\xi$, $K$ a compactly supported version of the heat kernel), we are left with solving~\eqref{e.dPD}. By Proposition \ref{pr.FP}, it is enough to show that for any $\gamma< \gamma_{dPD}=2\sqrt{2}-\sqrt{6}$, there exists a regularity $\alpha >-1$ s.t. the positive measure $\Theta:= \Wick{e^{\gamma \Phi}}$ belongs a.s. to $\calC_\s^\alpha$. Indeed given this regularity $\alpha>-1$, there exists $\kappa>0$ small enough such that $\beta:=\alpha+2 -\kappa \in (0,1)$ and $\alpha+\beta >0$. As $u(0)$ in Theorem \ref{th.mainT} is assumed to be $\Phi(0)+w$, with $w\in \calC^a(\T^2)$, we have $R:=\p_t \Phi - \frac 1 2 \Delta \Phi - \sqrt{2\pi} \xi$ ($R$ is smooth, see \cite{sine}), $v_0=w$ and $\Theta=\Wick{e^{\gamma \Phi}}$ which satisfy the assumptions of Proposition \ref{pr.FP} for a large enough $M>0$. This gives a fixed point solution to~\eqref{e.dPD} as well as~\eqref{e.DLSs} on the interval $[0,t(M)]$. Therefore in order to prove the existence of local solution to~\eqref{e.DLT}, 
it only remains to show the following Proposition, whose proof is postponed to the next Section.


\begin{proposition}\label{pr.reg}
Let us fix $\gamma <\hat \gamma_c=  2\sqrt{2}$. The positive measure $\Theta:=\Wick{e^{\gamma \Phi}}$ a.s. belongs to $\calC_\s^\alpha$ for any $\alpha < \bar \alpha(\gamma):=\frac {\gamma^2} 2 - 2 \sqrt{2} \gamma$. In particular for any $\gamma<\gamma_{dPD}=2\sqrt{2}-\sqrt{6}$, there exists $\alpha> -1$ s.t. $\Theta$ a.s. belongs to $\calC_\s^\alpha$. 
\end{proposition}



Now, we wish to prove the convergence in probability of $X_\eps(t,x)$ to $X(t,x)$ in $\calC_\s^\alpha$ as stated in Theorem \ref{th.mainT}. Recall that $X_\eps$ is solution to the SPDE~\eqref{e.DLT} which is driven by the regularized noise $\xi_\eps:= \varrho_\eps*\xi$. As in the continuous setting, it is natural to introduce the process $\Phi_\eps:= K* \xi_\eps$ where the kernel $K$ is the same as the one defining $\Phi$ in~\eqref{e.K}. As such, $\Phi_\eps$ is a solution of the linear heat equation up to a smooth correction term $R_\eps$:
\[
\p_t \Phi_\eps = \frac 1 2 \Delta \Phi_\eps  + R_\eps + \xi_\eps 
\] 
It can easily be shown that $R_\eps$ converges in probability to $R$ in say $\calC_\s^1$. See \cite{sine}. By the same Da Prato-Debussche change of variable, we are trying to find a solution $v_\eps$ to the system 
\begin{align}\label{e.dPDeps}
\begin{cases}
& \p_t v_\eps = \frac 1 2 \Delta v_\eps  - C_\varrho \eps^{\g2}  e^{\gamma \Phi_\eps} e^{\gamma v_\eps} - R_\eps \\
& v_\eps(t=0,x)= X_\eps(t=0,x)-\Phi_\eps(t=0,x) = w(x)
\end{cases}
\end{align}

We are thus exactly in the setting of Proposition \ref{pr.FP}. In particular, as we already know that $R_\eps \to R$, the continuity of the fixed point in the parameters $\Theta,v_0, R$ provided by this Proposition implies that it is enough for us to show the following result (whose proof is also postponed to Section \ref{s.RE}, see Theorem \ref{th.convergence}). 

\begin{proposition}\label{pr.conv}
For any $\gamma< \gamma_{L^2}=2$ and for any $t>0$, we have that $\Theta_\eps:= C_\varrho \eps^{\g2} e^{\gamma \Phi_\eps}$
converges in probability to $\Theta=\Wick{e^{\gamma X}}$ in $\calC_\s^\alpha(\Lambda_t)$ for any $\alpha<\g2 - 2\sqrt{2} \gamma$ (see Remark \ref{r.Wick} for the choice of $C_\varrho$). Furthermore the limit $\Theta=\Wick{e^{\gamma X}}$ does not depend on the mollifier $\varrho$. 
\end{proposition}

\begin{remark}\label{}
This result is interesting in its own for GMC measures as it provides a convergence in probability  which holds under a stronger topology (Besov spaces) than convergence results proved for GMC measures so far  (\cite{Kahane,DS,Nath}). See Theorem \ref{th.convergence} and the discussion which follows for a more precise convergence statement. 
\end{remark}

\begin{remark}
The value of $\gamma_{L^2}=2$ may look strange at first sight here. Indeed in $d=2$, the $L^2$-threshold happens at $\gamma=\sqrt{2}$, but in our present space-time setting, the $L^2$-threshold (the value of $\gamma$ at which $L^2$-moments cease to exist) happens at a later threshold $\gamma_{L^2}=2$. This is consistent with the moment computations in Lemma \ref{l.T0}. 
\end{remark}


\section{Regularity estimates and convergence of $\eps$-regularizations on the torus}\label{s.RE}

In this Section, we wish to prove Propositions \ref{pr.reg} and \ref{pr.conv}. As we shall see below, a significant part of the analytical work 
has been done already in \cite{sine}. (See for example Proposition \ref{l.key} below). 

\subsection{Method of proof}

The main ingredient to establish the regularity of the positive measure $\Theta:= \Wick{e^{\gamma X}}$ is an analysis of its moments. This is now classical in \textbf{Gaussian multiplicative chaos} and goes back to Kahane \cite{Kahane}, see also \cite{ChaosReview}. To obtain the regularity of $\Wick{e^{\gamma X}}$, one usually proceeds in three steps.
\bnum
\item First one shows that the measure $\Wick{e^{\gamma X}}$ when integrated, say, against a ball of radius one has $L^q$ moments until a certain threshold. Recall the threshold is $q_c=\frac {2d} {\gamma^2} = \frac 4 {\gamma^2}$ for Multiplicative chaos in $d=2$ dimensions. (See \cite{ChaosReview}). This will be the content of Lemma \ref{l.T0} below. As we shall see below, in our present space-time setting, the threshold will be instead $q_c=8/\gamma^2 (= 2(d+2)/\gamma^2)$. 
\item From the fact that the measure $\Wick{e^{\gamma X}}$ has {\em global} $L^q$ moments, one extracts quantitative $L^q$ estimates on the measure of small balls of radius $r$. In dimension $d=2$, this corresponds to the result (see \cite{ChaosReview}), 
\[
\Eb{M_\gamma(B(x,r)^q)} \asymp r^{\xi(q)}
\]
where $\xi(q)= (2+\frac {\gamma^2} 2)q - \frac {\gamma^2} 2 q^2$. It is sometimes called the \textbf{multifractal spectrum} of the measure $M_\gamma$. The classical proof of this estimates combines two elegant ingredients: \textbf{Kahane convexity inequality} as well as an appropriate notion of invariance: \textbf{log-$*$-invariance} (see \cite{ChaosReview}). As we shall see below, we will not be able to rely on this exact scale-invariance as we are not aware of the existence of log-$*$-invariant kernels in space-time $\R\times \R^2$. 

\item Finally, it is not hard from the knowledge of the multi-fractal spectrum $q\mapsto \xi(q)$ to deduce the regularity of $\Wick{e^{\gamma X}}$ by a Kolmogorov argument. This is done in the multiplicative Chaos setting for example in \cite{LBM} (see v2 on Arxiv) and is reminiscent of the more general Theorem 1.1 in \cite{tightness}.  See Theorem \ref{th.convergence}. 
\enum

Examples of positive-definite kernels which are log-$*$-invariant are known on each euclidean $\R^d, d\geq 1$. If the dimension $d\in \{1,2,3\}$, then, see \cite{ChaosReview}, one can take 
\[
K(x):= \log_+ \frac 1 {\|x\|_2}
\]
(If $d=3$, it is not known whether it is of $\sigma$-positive type in the sense of Kahane or not). 
If $d\geq 4$, this kernel is no-longer positive definite. Yet, as argued in \cite{ChaosReview}, the following kernel 
\[
K(x):= \int_{\S^{d-1}\subset \R^d} \log_+ \frac 1 {|\<{x,s}|} \sigma_{\S^{d-1}}(ds)
\]
gives an example of log-$*$-kernel in $\R^d$. This nice extension does not apply to our parabolic setting where space and time scale differently. We are not aware of any log-$*$-invariant kernel in the parabolic setting and will then have to proceed (slightly) differently. 

\subsection{Estimates on the log-correlated field $\Phi, \Phi_\eps$} $ $

Let us first introduce some notations. Recall $\Phi$ is an (approximate) solution to the linear heat equation defined as $K*\xi$ (see \eqref{e.K}). 
Let us now define our (approximate) solution to the linear heat equation driven by $\xi_\eps:= \varrho_\eps * \xi$. We define $\Phi_\eps$ to be the Gaussian Field
\begin{align}\label{e.Phieps}
\Phi_\eps: = K * \xi_\eps 
\end{align}
As in \cite{sine}, it can be seen that the correction term to the heat equation $R_\eps := \p_t \Phi_\eps - \frac 1 2 \Delta \Phi_\eps - \xi_\eps$ is a smooth function which converges as $\eps\to 0$ to a smooth limiting  function $R$.  
By the associativity of convolution, this Gaussian process can also be written as $K_\eps * \xi$ where one introduces the kernel $K_\eps:= K* \varrho_\eps$. Finally, 
 its covariance kernel $\calQ_\eps(t,x) := \Eb{\Phi_\eps(0,0) \Phi_\eps(t,x)}$  is given by (see \cite{sine})
\begin{align*}\label{}
\calQ_\eps
& = K_\eps * \calT(K_\eps)\\
& = \calQ * (\varrho_\eps * \calT \varrho_\eps)\, \;\;\;\;\;\;\; \calQ= K * \calT K\,,
\end{align*}
where $\calT f(t,x):=f(-t,-x)$.

We first collect very detailed estimates on the covariance structure of $\Phi,\Phi_\eps$ which were obtained in \cite{sine}.
These key estimates are summarised in the following Proposition.

\begin{proposition}[Section 3. in \cite{sine}]\label{l.key}
The covariance kernels $\calQ$ and $\calQ_\eps$ statisfy the following properties on $\T^2$. 
\bnum
\item 
There exists a constant $\hat C_\varrho$ which only depends on the smoothing function $\varrho$ such that 
\begin{align}\label{e.Ess0}
\calQ_\eps(0) (=\calQ_\eps(0,0,0)) = \log \frac 1 \eps +  \hat C_\varrho + O(\eps^2)  
\end{align}
\item There exist $c,C$ which are $\varrho$-dependent s.t. for any $\eps\in (0,1]$ and any $z,\, \|z\|_\s \leq 1$, 
\begin{align}\label{e.Ess1eps}
- \log(\| z\|_\s + \eps) + c \leq  \calQ_\eps(z) \leq - \log(\| z\|_\s + \eps) + C
\end{align}
\item 
\begin{align}\label{e.Ess1}
\calQ(z) \sim - \log(\| z\|_\s)
\end{align}
\item 
\begin{align}\label{e.Ess2}
|\calQ_\eps(z) - \calQ(z)| \leq C  \left( \frac \eps {\|z\|_\s} \wedge (1+ \log \frac \eps {\| z\|_\s} )\right)
\end{align}
\enum
N.B. Item (2) follows from the estimate (3.7) in \cite{sine}. 
\end{proposition}

\subsection{Moments computation} 
In this subsection (as in Section 3 in \cite{sine}), instead of using $\Theta_\eps:= C_\varrho \eps^{\g2} e^{\gamma \Phi_\eps}$ it will be more convenient to rely on the exact Wick ordering 
\[
\Theta_\eps:= \Wick{e^{\gamma \Phi_\eps}}:= e^{\gamma \Phi_\eps - \g2 \calQ_\eps(0)}\,.
\]
and $\Theta$ will denote its limit (yet to be proved) in the regularity spaces $\calC_\s^\alpha$ defined in Subsection \ref{ss.Besov}. 
One can then easily get back to the original choice of renormalization in Theorem \ref{th.mainT} thanks to the above estimate~\eqref{e.Ess0}.

\begin{proposition}\label{pr.moments}
$ $

\bnum
\item For any $0\leq \gamma < \hat \gamma_c = 2\sqrt{2}$, and any (real number) $q< \frac 8 {\gamma^2}$, then as $\eps \to 0$, 
\begin{align}\label{e.Lq}
\Eb{|\<{f_z^\lambda, \Theta_\eps}|^q} \lesssim  \lambda^{-\g2 q(q-1)}\,,
\end{align}
uniformly over all test functions $f$ bounded by 1 and supported in the unit ball (for $\|\cdot\|_\s$), all $\lambda\in (0,1]$ and 
all space-time points $z$. 

\item Furthermore if $\gamma<2$, for $\kappa=\kappa(\gamma)$ sufficiently small, one has 
\begin{align}\label{e.L2}
\Eb{|\scal{f_z^\lambda, \Theta_\eps - \Theta_{\bar \eps}}|^2} \lesssim (\eps\vee \bar \eps)^{2\kappa} \lambda^{-2\kappa- \gamma^2 }\,,
\end{align}
uniformly in the same $f,\lambda, z$. 
\enum
\end{proposition}

\begin{remark}\label{}
See Appendix \ref{a.1} for a different proof of this result in the special case of integer moments $q\in \N < \frac 8 {\gamma^2}$. The proof in appendix \ref{a.1} is much closer to the analysis in \cite{sine} of the polynomial moments for Sine-Gordon.
\end{remark}

\ni
\textit{Proof.}

\ni
{\bf Second moment estimate.}
The proof of the estimate (2) is exactly the same one as the proof of the estimate (3.12) in \cite{sine}. Indeed, the latter paper analyses the complex valued distribution $\Psi_\eps:= \Wick{e^{i\beta \Phi_\eps}} = e^{i \beta \Phi_\eps \purple{+} \frac{\beta^2} 2 \calQ_\eps(0)}$ instead of our positive measure $\Theta_\eps$. The second moment which is analyzed in \cite{sine} is 

\begin{align*}\label{}
\Eb{|\<{\varphi_x^\lambda, \Psi_\eps - \Psi_{\bar \eps}}|^2}
& = \iint \varphi_x^\lambda(y) \varphi_x^\lambda(y+z) \Eb{ [\Psi_\eps - \Psi_{\bar \eps}](y)  [\closure{\Psi_\eps} - \closure{\Psi_{\bar \eps}}](y+z)} dydz 
\end{align*}
It turns out it matches exactly the second moment we need. Indeed, 
\begin{align*}\label{}
\Eb{\Psi_\eps(y) \closure{\Psi_{\bar \eps}}(y+z)} 
& = \exp(\frac {\beta^2} 2 (\calQ_\eps(0) + \calQ_{\bar \eps}(0) )) \Eb{e^{i \beta \Phi_\eps(y) -i\beta \Phi_{\bar \eps}(y+z)}}  \\
&=   \exp(\frac {\beta^2} 2 (\calQ_\eps(0) + \calQ_{\bar \eps}(0) ))  \exp(-\beta^2/2 \Var{\Phi_\eps(y)-\Phi_{\bar\eps}(y+z)}) \\
& = \Eb{\Theta_\eps(y) \Theta_{\bar \eps}(y+z)}\,.
\end{align*}

\medskip
\ni
{\bf $q^{th}$ moments.}
Now for the moment estimate (1), this a different story. Indeed for Sine-Gordon (see Theorem 3.2. in \cite{sine}), one has instead $\Eb{|\<{f_z^\lambda, \Psi_\eps}|^q} \lesssim \lambda^{-\frac {\beta^2 q}{4\pi}}$. In particular the moments behave \textbf{linearly in $q$} which is a sign of {\bf monofractality}. Here, we have a \textbf{non-linear spectrum} which is the signature of intermittency.  

Let us first show that $\int_{[-T,T]\times \T^2} \Theta_\eps$ has $L^q$ moments uniformly in $\eps>0$ for all $q< 8 / \gamma^2$. (see Lemma \ref{l.T0} below). In order to avoid dealing with the periodic boundary issues inherent to $\T^2$, it will be enough (by using the $L^q$-triangle inequality) to stick to a sub-domain $D:=[0,1] \times [0,1]^2 \subset [0,1]\times \T^2$ and to show that $\int_{D} \Theta_\eps$ has uniform $L^q$ moments uniformly in $\eps>0$. (On the sphere this step will be even more needed to avoid spherical boundary conditions and to reduce curvature effects). Namely, we start by proving 

\begin{lemma}\label{l.T0}
For any $\gamma<\hat \gamma_c = 2 \sqrt{2}$ and any $q<8/\gamma^2$,
\begin{align*}\label{}
\sup_{\eps>0} \Eb{(\int_{D=[0,1]\times[0,1]^2} \Theta_{\eps})^q} <\infty
\end{align*}
\end{lemma}

These types of moment estimates go back to \cite{Kahane, Mandelbrot} and we will follow very closely here the proof in \cite[Appendix D]{Bacry}.
The main technical tool is the following convexity inequality from Kahane. 
\begin{lemma}[Kahane convexity inequality, \cite{Kahane}, see also \cite{ChaosReview}]\label{l.K}
Let $(X_i)_{1\leq i \leq n}$ and $(Y_i)_{1\leq i \leq n}$ two centred Gaussian vectors which satisfy for any $i,j$,
\[
\Eb{X_i X_j} \leq \Eb{Y_i Y_j}\,.
\]
Then, for any combinations of nonnegative weights $(p_i)_{1\leq i \leq n}$ and any convex function $F: \R_+ \to \R$,
\begin{align*}\label{}
\Eb{F(\sum_{i=1}^n p_i e^{X_i - \frac 1 2 \Eb{X_i^2}})} \leq \Eb{F(\sum_{i=1}^n p_i e^{Y_i - \frac 1 2 \Eb{Y_i^2}})}\,.
\end{align*}
The property extends immediately to smooth centred Gaussian processes, say in our case $\{X_\eps(z)\}_{z\in \R \times \T^2}$, $\{Y_\eps(z)\}_{z\in \R\times \T^2}$. If for any $z,z'$, one has 
\[
\Eb{X_\eps(z) X_\eps(z')} \leq \Eb{Y_\eps(z) Y_\eps(z')}\,.
\]
Then, for any finite measure $\mu$ on $\R \times \T^2$:
\begin{align*}\label{}
\Eb{F(\int e^{\gamma X_\eps(z) - \frac {\gamma^2} 2 \Eb{X_\eps(z)^2}} d\mu(z) )} \leq 
\Eb{F(\int e^{\gamma Y_\eps(z) - \frac {\gamma^2} 2 \Eb{Y_\eps(z)^2}} d\mu(z) )}\,.
\end{align*}
\end{lemma}

Let us fix some $\delta>0$. Divide $D=[0,1]\times [0,1]^2$ into 8 sets of tiles $\calQ_1,\ldots, \calQ_8$ where the tiles are translates of the parabolic tile $Q_0:=[0,\delta^2]\times [0,\delta]^2$. The sets $\calQ_1,\ldots,\calQ_8$ are chosen so that tiles in each $\calQ_i$ are at $\| \cdot\|_\s$-distance at least $\delta$. To be more explicit, one may choose $\calQ_1:= \bigcup_{0\leq i,j\leq \floor{\delta^{-1}/2}, 0\leq k \leq \floor{\delta^{-2}/2}} (2k\delta^2,2i\delta,2j\delta)+Q_0$ and then define $\calQ_m,m\geq 2$ as translates of $\calQ_1$ in order to cover $[0,1]\times [0,1]^2$.
We have for each $q \geq 1$, 
\begin{align*}\label{}
\Eb{(\int_{[0,1]\times \T^2} \Theta_\eps )^q}^{1/q} \leq \sum_{i=1}^{8} \Eb{(\int_{\calQ_i} \Theta_\eps)^q}^{1/q}\,.
\end{align*}
Let us now focus on one of the $\calQ_i$, say $\calQ_1$. We wish to use Kahane's convexity inequality to introduce some independence as well as scaling into the analysis. As discussed at the beginning of this subsection, we do not have explicit log-*-scale invariant kernels at our disposal (such as $\log_+$ on $\R^2$ in this parabolic setting), we will thus introduce the following fields:
\begin{definition}\label{}
Fix $\delta\in (0,1]$ and $\eps>0$. 
Let $Q_0$ be the parabolic tile $[0,\delta^2]\times [0,\delta]^2$. Let $Z_{\delta \eps}$ be the centred Gaussian field on $Q_0$ defined by 
\begin{align*}\label{}
(Z_{\delta \eps}(z))_{z\in Q_0}:= (\Phi_\eps(\frac 1 \delta \cdot z))_{z \in Q_0}
\end{align*}
where $\lambda \cdot (t,x) : = (\lambda^2 t, \lambda x)$ and  $\Phi_\eps$ is the Gaussian field defined earlier in~\eqref{e.Phieps} by $K*\xi_\eps$.  (We have chosen the subscript $\delta \eps$ here because the process $Z_{\delta \eps}$ corresponds to a field which is regularised at scale $\delta \eps$.)

Assign for each tile $Q$ in $\calQ_1$, an iid copy $Z_{\delta \eps}^Q \sim Z_{\delta \eps}$ and consider the field 
\begin{align*}\label{}
Z_{\delta \eps}^{\calQ_1}(z):= \sum_{Q\in \calQ_1} 1_{z\in Q} \; Z_{\delta \eps}^Q(z) 
\end{align*}
\end{definition}

From the properties listed in Proposition \ref{l.key}, one obtains the following estimate. 
\begin{lemma}\label{l.KahBound}
There exists some finite $K>0$ s.t. for any $(z,z') \in \calQ_1$, $\delta\in (0,1]$ and $\eps>0$, 
\begin{align}\label{e.KahBound}
\Eb{\Phi_{\delta \eps}(z) \Phi_{\delta \eps}(z')} \leq  \Eb{Z_{\delta \eps}^{\calQ_1}(z) Z_{\delta \eps}^{\calQ_1}(z')} +  \log \frac 1 \delta + K\,.
\end{align}
\end{lemma}
\ni
{\em Proof.}
Let us first deal with the case where $z,z'$ belong to the same $\delta$-tile $Q$ in $\calQ_1$.
Using the  above definition of $Z_{\delta \eps}^Q$ and the estimate~\eqref{e.Ess1eps}, one has for any $(z,z') \in Q$
\begin{align*}\label{}
\Eb{Z_{\delta \eps}^{\calQ_1}(z) Z_{\delta \eps}^{\calQ_1}(z')} &= \Eb{\Phi_\eps(\frac 1 \delta \cdot z)  \Phi_\eps(\frac 1 \delta \cdot z')} \\
& \geq \log \frac 1 {\| \frac 1 \delta \cdot (z - z') \|_\s + \eps} + c \\
& \geq  \log \frac 1 {\| z - z' \|_\s + \delta \eps}  + \log \delta + c \\
& \geq \Eb{\Phi_{\delta \eps}(z) \Phi_{\delta \eps}(z')} + \log \delta + c-C 
\end{align*}
Now, if $z,z'$ are in $\calQ_1$ but do not belong to the same tile $Q$, we have by independance of $Z_{\delta \eps}^{\calQ_1}$ on different tiles that $\Eb{Z_{\delta \eps}^{\calQ_1}(z) Z_{\delta \eps}^{\calQ_1}(z')}=0$. Furthermore, by construction of $\calQ_1$, we must have $\|z-z'\|_\s > \delta$ which implies $\Eb{\Phi_{\delta \eps}(z) \Phi_{\delta \eps}(z')} \leq \log \frac 1 \delta + K$ for some constant $K>0$.\qed

\ni
{\em Proof of Lemma \ref{l.T0} continued.}
Let $\Omega_{\log \frac 1 \delta + K}$ be a global normal r.v. of variance $\log \frac 1 \delta + K$ independent of $Z_{\delta \eps}$. Introduce now the Gaussian field on $\calQ_1$, $Y_{\delta \eps}(z):= Z_{\delta \eps}^{\calQ_1}(z) + \Omega_{\log \frac 1 \delta + K}$ so that one has for any $z,z'\in \calQ_1$, $\Eb{\Phi_{\delta\eps}(z) \Phi_{\delta\eps}(z')} \leq \Eb{Y_{\delta \eps}(z) Y_{\delta \eps}(z')}$. One may now use the above Kahane convexity inequality for $q\geq 1$:
\begin{align}\label{e.DDD}
\Eb{(\int_{\calQ_1} \Theta_{\delta \eps})^q} & = \Eb{(\int_{\calQ_1} e^{\gamma \Phi_{\delta \eps}(z) - \g2 \calQ_{\delta\eps}(0)})^q} \nn \\
& \leq  C  \delta^{\g2(q-q^2)} \Eb{(\int_{\calQ_1} e^{\gamma Z_{\delta \eps}^{\calQ_1}(z) - \g2 \calQ_{\eps}(0)})^q}\,,
\end{align}
where the term $\delta^{\frac{\gamma^2} 2 q}$ follows from the fact that by ~\eqref{e.Ess0} $\calQ_{\delta\eps}(0) = \log\frac 1 \delta + \log \frac 1 \eps \pm O(1)$ while $\delta^{-\frac{\gamma^2} 2 q^2}$ corresponds to the expectation of $\Eb{e^{\gamma q \Omega_{\log \frac 1 \delta +K}}}$.
Now, following almost verbatim the appendix D in \cite{Bacry} (to which we refer) and using the same notations, let $n\geq 2\in \N$, s.t. $n-1< q \leq n$, 
we use the sub-additivity of $x\mapsto x^{q/n}$ as follows, 
\begin{align}\label{e.DnonD}
\Eb{(\int_{\calQ_1} e^{\gamma Z_{\delta \eps}^{\calQ_1}(z) - \g2 \calQ_{\eps}(0)})^q}
& = \Eb{(\sum_{Q\in \calQ_1} \int_{Q} e^{\gamma Z_{\delta \eps}^{Q}(z) - \g2 \calQ_{\eps}(0)})^q} \nn \\
& \leq \Eb{\Bigl(\sum_{Q\in \calQ_1}  (\int_Q e^{\gamma Z_{\delta \eps}^{Q}(z) - \g2 \calQ_{\eps}(0)})^{\frac q n} \Bigr)^n}
\end{align}
By expanding the $n$-th power, we deal separately with the diagonal and non-diagonal term. Let us start with the latter one. An important feature of our field $Z_{\delta\eps}^{\calQ_1}$ is that it is made of iid copies $Z_{\delta\eps}^Q$ for each $Q\in \calQ_1$. Using this independance, each non-diagonal term is of the form $\prod_{j=1}^M \Eb{ (\int_{Q_j} e^{\gamma Z_{\delta \eps}^{Q}(z) - \g2 \calQ_{\eps}(0)})^{\frac {q s_j} n} }$ where $M$ is the number of parabolic $\delta$-tiles in $\calQ_1$, i.e. $M$ is of order $\frac 1 {\delta^4}$ and $\{s_j\}$ are integers in $\{0,\ldots, n-1\}$ which sum up to $n$. 
Using the fact that $n-1 < q \leq n$, $s_j \leq n-1$ and $\sum_j s_j =n$, each of these non-diagonal terms are bounded by  
\begin{align*}\label{}
\prod_{j=1}^M \Eb{ (\int_{Q_j} e^{\gamma Z_{\delta \eps}^{Q_j}(z) - \g2 \calQ_{\eps}(0)})^{\frac {q s_j} n} }
& \leq \prod_{j=1}^M \Eb{ (\int_{Q_j} e^{\gamma Z_{\delta \eps}^{Q_j}(z) - \g2 \calQ_{\eps}(0)})^{n-1}}^{\frac {q s_j} {n(n-1)}} \\
& =  \Eb{ (\int_{Q_0} e^{\gamma Z_{\delta \eps}(z) - \g2 \calQ_{\eps}(0)})^{n-1}}^{\frac q {n-1}}
\end{align*}
Now, by the definition of $Z_{\delta \eps}$ which is nothing but a rescaling on $Q_0=[0,\delta^2]\times[0,\delta]^2$ of $\Phi_\eps$ on $D=[0,1]\times [0,1]^2$, we have by change of variable that 
\begin{align*}\label{}
\Eb{ (\int_{Q_0} e^{\gamma Z_{\delta \eps}(z) - \g2 \calQ_{\eps}(0)}dz)^{n-1}}^{\frac q {n-1}}
&= \delta^{4q}  \;\; \Eb{ (\int_{D} \Theta_\eps(z) dz)^{n-1}}^{\frac q {n-1}}
\end{align*}

As there are at most $M^n= O(\delta^{-4n})$ such non-diagonal terms in ~\eqref{e.DnonD}, they contribute at most 
\begin{align}\label{e.nonD}
C \delta^{\g2(q-q^2)}  \delta^{4(q-n)}   \;\; \Eb{ (\int_{D} \Theta_\eps(z) dz)^{n-1}}^{\frac q {n-1}}
\end{align}
to the upper bound~\eqref{e.DDD}.  Now, the $M=O(\frac 1 {\delta^4})$ diagonal terms in~\eqref{e.DnonD} all have the same expression 
\begin{align*}\label{}
\Eb{ (\int_{Q_0} e^{\gamma Z_{\delta \eps}(z) - \g2 \calQ_{\eps}(0)}dz)^q}
&= \delta^{4q} \;\; \Eb{ (\int_{D} \Theta_\eps(z) dz)^q}
\end{align*}
Combining into~\eqref{e.DnonD} these $O(\frac 1 {\delta^4})$ diagonal terms with the contribution of the non-diagonal ones~\eqref{e.nonD}, we obtain 
\begin{align*}\label{}
 \Eb{(\int_{D} \Theta_{\delta \eps})^q}
& \\
& \hskip -2 cm  \leq 
\tilde C \delta^{\g2(q-q^2) + 4q-4} \Eb{ (\int_{D} \Theta_\eps(z) dz)^q} + \tilde C \delta^{\g2(q-q^2)}  \delta^{4(q-n)}   \;\; \Eb{ (\int_{D} \Theta_\eps(z) dz)^{n-1}}^{\frac q {n-1}}
\end{align*}
From now on, we conclude the proof of Lemma \ref{l.T0} in two steps. 
\bnum
\item First one assumes that we have for granted that $\sup_{\eps>0} \Eb{(\int_{D} \Theta_{ \eps})^{n-1}}<\infty$. Note that if $1<q<8/\gamma^2$, the exponent $\g2(q-q^2) + 4q -4>0$ which means that one can choose $\delta$ small enough so that $\tilde C \delta^{\g2(q-q^2) + 4q}<1$. For particular choices of regularisations for which $\eps \to (\int_{D} \Theta_{\eps})^q$ is a sub-martingale we conclude immediately. If we are not in the case of an exact sub-martingale, we can still conclude easily by noticing that for $\delta$ small enough, $\Eb{(\int_{D} \Theta_{\delta \eps})^q} \leq \frac 1 2 \Eb{(\int_{D} \Theta_{\eps})^q} + K$ for a large enough constant $K$ (where we used the above assumption). This implies that $\sup_{k\geq 0} \Eb{(\int_{D} \Theta_{\delta^k \eps})^q}<\infty$ and it is straightforward for example by varying the iteration parameter $\delta$ to deduce Lemma \ref{l.T0}. 

\item It thus remains to show our assumption that $\sup_{\eps>0} \Eb{(\int_{D} \Theta_{\eps})^{n-1}}<\infty$. This can be easily achieved by running this proof inductively on $n$: if $q\in (1,2]$, then $n-1=1$ and the assumption is obvious. In particular item (1) propagates the assumption to $n-1=q=2$ and one can keep iterating for $q\in(2,3]$ and so on. (See Appendix D in \cite{Bacry}). 
\enum \qed

\begin{remark}\label{}
Note that Kahane convexity inequality is not even needed in \cite[Apendix D]{Bacry}, but it seems less clear how to avoid it here due to the lack of known exact $*$-scale invariance kernel in our parabolic setting. 
\end{remark}

To conclude the proof of Proposition \ref{pr.moments}, it is enough to prove the following so-called \textbf{multifractal spectrum} which describes the $q^{th}$ moments of small balls. 
\begin{proposition}\label{pr.mom}
Let us fix $\gamma<\hat \gamma_c = 2\sqrt{2}$. For any $0\leq q < \frac 8 {\gamma^2}$, 
there exists $C<\infty$ s.t. for any $r\leq 1$ and any $z\in \R\times \T^2$, 
\begin{align*}\label{}
\sup_{\eps>0}\Eb{(\int_{B_\s(z,r)} \Theta_\eps)^q} \leq C r^{\xi_\s(q)}
\end{align*}
where $B_\s(z,r)$ is the $r$-ball around $z$ for the parabolic distance $\|\cdot\|_\s$ and $\xi_\s(q)=\g2(q-q^2)+4q$. 
\end{proposition}

\ni
{\em Proof.}
Clearly $B_\s(z,r)$ is included in a translate of a tile $Q_0:=[0,\delta^2]\times[0,\delta]^2$ used in the previous proof with $\delta=2r$. Now if $q\geq 1$, the $L^q$ moment of $\int_{Q_0} \Theta_{\delta\eps}$ has been bounded from above in the previous proof (this corresponds to each of the "diagonal terms") by $\tilde C \delta^{\xi_\s(q)} \times \Eb{ (\int_{D} \Theta_\eps(z) dz)^q}$. Using Lemma \ref{l.T0}, this implies readily that $\sup_{\eps<r}\Eb{(\int_{B_\s(z,r)} \Theta_\eps)^q} \leq O(1) r^{\xi_\s(q)}$. It remains to handle regularisations $\eps$ which are coarser than the radius $r$. The idea in such a case is to use the fact that the field $\Phi_\eps$ is smooth within $B_\s(z,r)$. One can then dominate the field simply by $\sup_{u\in B_\s(z,r)} \Phi_\eps(u)$. To make the analysis quantitative one may write $\Phi_\eps(u)=\Phi_\eps(z)+Y_\eps(u)$ and use the fact that the Gaussian process $Y_\eps(u)$ does not fluctuate much within $B_\s(z,r)$. One can then rely on standard Gaussian concentration bounds, which lead us to an upper bound $O(1)  \eps^{\g2(q-q^2)} r^{4q}$. This gives us a better bound than our desired upper bound when $0\leq q<8/\gamma^2$ and thus concludes the proof.  
We refer to \cite{sphere} and in particular to the end of the proof of Lemma 3.3,  where the details of such a Gaussian domination are given in a very similar setting. 
The case $q\in [0,1)$  can be done in the same fashion except $x\mapsto x^q$ is now concave and one has to use the estimate~\eqref{e.Ess1eps} the other way around to conclude via Kesten's convexity inequality. 
\qed

\subsection{Convergence of $\eps$-regularizations.}

From the estimates given in Proposition \ref{pr.moments}, we deduce the following convergence result (which is a detailed version of Proposition \ref{pr.conv}). 

\begin{theorem}\label{th.convergence}
For any $0\leq \gamma < 2$, there exists a stationary random positive measure $\Theta$ which is independent of the mollifier $\varrho$ and which satisfies for any $q\in [0, \frac 8 {\gamma^2})$ and any fixed $\kappa>0$, 
\begin{align}\label{e.diff}
\Eb{|\<{f_z^\lambda, \Theta - \Theta_\eps}|^q} \lesssim \eps^\kappa \lambda^{-\g2 q(q-1)}\,,
\end{align}
uniformly over all test functions $f$ bounded by 1 and supported in the unit ball (for $\|\cdot\|_\s$), all $\lambda\in (0,1]$, and locally uniformly over space-time points $z$. 
\medskip

It follows that $\Theta_\eps$ converges in probability to $\Theta$ in the space $\calC_\s^{\alpha}$ for any $\alpha< \g2 -2\sqrt{2} \gamma$. 
\end{theorem}

\begin{remark}\label{}
As mentioned above, it gives a different proof (inspired by \cite{sine}) of convergence of these measures as $\eps \to 0$ independently of mollifiers. Also the topology is slightly stronger than in works  on Gaussian Multiplicative Chaos so far. On the other hand, as a drawback, it only works  below the $L^2$ threshold as the proof is based on the estimate~\eqref{e.L2}. 
\end{remark}

\ni
{\em Proof of Theorem \ref{th.convergence}.}

The first part of the Theorem is proved exactly as the L.H.S of (3.3) in Theorem 3.2 in \cite{sine}. The fact it implies the convergence in $\calC_\s^\alpha$ for all $\alpha< \g2 - 2\sqrt{2} \gamma$ would follow the same lines as the proof of Theorem 2.1 in \cite{sine} (namely a characterization of $\calC_\s^\alpha$ in terms of wavelet coefficients and a Kolmogorov continuity argument, see Sections 3 and 10 in \cite{HairerReg}) if the moment bound~\eqref{e.diff},~\eqref{e.Lq} would hold for all integer $q$ and would be linear in $q$. A refined version of the link with wavelets is provided by Theorem 1.1. in \cite{tightness}. It states that a family of linear forms on $C_c^\infty(\R^d)$, $\{\eta_m\}_{m\geq 1}$ is tight in $\calC^\alpha(\R^d)$ for every $\alpha<\beta -\frac d q$ if for every wavelet $\psi$ in a suitably chosen finite collection $\Psi$,  
\[
\sup_{x\in \R^d} \lambda^{-d} \Eb{|\<{\eta_m, \psi(\lambda^{-1}(\cdot -x)}|^q}^{1/q} \leq C \lambda^\beta\,.
\] 
See \cite{tightness} for wavelet notations as well as a more detailed statement. It is straightforward to extend their statement to the parabolic setting (as in \cite[Sections 3,10]{HairerReg}). Their statement reads as follows in this parabolic setting: a family $\{\eta_m\}_{m\geq 1}$ of linear transforms on $C_c^\infty(\R \times \T^2)$ is tight in $\calC_\s^\alpha$ for every $\alpha < \beta - \frac 4 q$ if for every $\psi$ in a (suitably chosen) finite collection $\Psi$,
\[
\sup_{z\in \R \times \T^2} \Eb{|\<{\eta_m, \psi_z^\lambda}|^q}^{1/q} \leq C \lambda^\beta\,,
\] 
where $f \mapsto f_z^\lambda$ is the parabolic rescaling of $f$ defined earlier. Using the estimate ~\eqref{e.Lq}, this implies a tightness result in $\calC_\s^\alpha$, for all $\alpha < -\g2 (q-1) - \frac 4 q$ as far as $q<8/\gamma^2$. It is not hard by proceeding as in Sections 3 and 10 in \cite{HairerReg} to upgrade their tightness statement to a convergence statement by relying on~\eqref{e.diff}. Now optimizing in $q\in [0, 8/\gamma^2)$ gives $\hat q=\frac {2\sqrt{2}} \gamma$ and leads to a convergence in $\calC_\s^\alpha$ for all $\alpha<\g2 - 2\sqrt{2} {\gamma}$ which is sharp. (See Remark \ref{r.sharp} below for the explanation why this is sharp). 
This ends the proof of Theorem \ref{th.convergence} which implies Proposition \ref{pr.conv}.
\qed

\begin{remark}\label{r.sharp}
It is not hard to see that the convergence cannot hold in spaces of higher regularity than the threshold $\g2 - 2\sqrt{2} \gamma$. This can be proved for example by analysing the structure of space-time \textbf{thick points} of the log-correlated field $\Phi$ in space/time. 
\end{remark}

\ni
{\em Sketch of proof of Proposition \ref{pr.reg}.}\\
\ni
The point of Proposition \ref{pr.reg} is to go beyond the $L^2$ threshold. The techniques of \cite{Kahane, Shamov, Nath} all apply easily to our present setting and we shall not give more details here. Let us stress though that as opposed to the above Theorem \ref{th.convergence}, we do not claim a convergence in Besov spaces $\calC_\s^\alpha$ here as the crucial $L^2$ estimate~\eqref{e.L2} is no longer valid. The techniques of \cite{Kahane, Shamov, Nath}  only imply a convergence result under the weak topology. The following more precise statement than Proposition \ref{pr.reg} follows from these techniques.

\begin{proposition}\label{}
For all $\gamma < 2\sqrt{2}$, the positive measure $\Theta_\eps$ converges weakly  a.s. to a non-degenerate positive measure $\Theta$.  The limiting measure $\Theta$ does not depend on the mollifier $\varrho$ and a.s. belongs to $\calC_\s^\alpha$ for all $\alpha < \g2 - 2\sqrt{2} \gamma$ (note that we do not claim nor need the convergence in this space when $\gamma\geq 2$). 
\end{proposition}

\section{Beyond Da Prato-Debussche by exploiting the positivity}\label{s.POS}

In this Section we prove Theorem \ref{th.main2} which extends the local existence result (but not the convergence result) to higher values of $\gamma$ than the Da Prato-Debussche threshold $\gamma_{dPD}:=2\sqrt{2} - \sqrt{6}$ by exploiting the positivity of the measure $\Theta = \Wick{e^{\gamma \Phi}}$. The main step where one can rely on the positivity is when one needs to define a product between rough objects. Indeed, as stated in \cite{ICM}, 
 Young's Theorem \ref{th.mult} {\em yields a sharp criterion for when, in the absence of any other structural knowledge, one can multiply a function and  distribution of prescribed regularity.} In our present situation, we do have additional structural knowledge as we are multiplying functions with positive measures rather than general distributions. Let us then strengthen Theorem \ref{th.mult}  as follows in this particular case (for clarity, we write a statement as close as possible to Theorem \ref{th.mult}). 

\begin{proposition}\label{pr.multPOS}
Suppose only $\beta>0$ and $\alpha <0$ (as opposed to $\alpha+\beta>0$ in Theorem \ref{th.mult}), then there exists a  bilinear form $B(\cdot, \cdot): (\calC_\s^\alpha\cap M^+([0,T] \times \T^2)) \times \calC_\s^\beta \to \calC_\s^{\alpha}$ satisfying \footnote{$M^+(T)$ denotes here the space of positive measures on $T$ and we view positive measures as distributions acting on smooth functions. }
\bi
\item[i)] $B(f(x)dx,g)$ coincides with the classical product when $f,g$ are smooth. 
\item[ii)] There exists $C>0$ s.t. for any $\mu \in \calC_\s^\alpha\cap M^+([0,T] \times \T^2), g\in \calC_\s^\beta$, 
\[
\|B(\mu,g)\|_{\calC_\s^{\alpha}} \leq C \| \mu \|_{\calC_\s^\alpha} \| g \|_{\calC_\s^\beta}
\]
\ei
N.B. Note that we do not claim a  continuous extension in the first argument here (the continuity in the second argument is clear from $(ii)$). In particular we do not have a control of the form $\|B(\mu,g) - B(\tilde \mu,g)\|_{\calC_\s^{\alpha}} \leq C \| \mu -\tilde \mu \|_{\calC_\s^\alpha} \| g \|_{\calC_\s^\beta}$. This is the reason why we do not obtain a convergence result as $\eps \searrow 0$. 
\end{proposition}

\ni
\textit{Proof.}
First, the extension is nothing but the fact that as $[0,T]\times \T^2$ is compact, positive measures form the dual of continuous functions on $[0,T] \times \T^2$ and can then be tested against any test function in $\calC_\s^\beta$ as far as $\beta>0$ which we assumed. It clearly corresponds to the classical product  when $\mu(dx)=f(x) dx$ and $f$ and $g$ are smooth. Now the second item follows easily  by noticing that with $m=\ceil{-\alpha} \geq 1$, one has (see Subsection \ref{ss.Besov} for notations),
\begin{align*}\label{}
\sup_{z \in [0,T]\times \T^2}\sup_{\eta \in B_m, \lambda \in (0,1]} \frac {|\int_{B(z,\lambda)} \eta_x^\lambda(u) g(u) \mu(du)|} {\lambda^{\alpha}} 
& \leq C  \| g \|_\infty  \sup_z \sup_{\lambda>0} \frac {\int_{B(z,\lambda)} \mu(du)} {\lambda^{\alpha+4}} \\
& \leq C \|g \|_\infty \| \mu \|_{\Cs^\alpha}
\end{align*}
where we used the positivity in the first inequality.  \qed 
\medskip

\ni
\textit{Proof of Theorem \ref{th.main2}.}
It is enough here to follow almost verbatim the proof in Subsection \ref{ss.settle} which settles the fixed point argument. The only difference here is that the map defined in Definition \ref{d.FPmap},
\begin{align*}
F_{t,\Theta, v_0,R} \,\, : \,\, u \mapsto \Bigl( (s,x)\in \Lambda_t \mapsto K(- \Theta e^{\gamma u} - R)(s,x) + P(v_0)(s,x)\Bigr)
\end{align*}
is now considered with an input $\Theta \in \calC_\s^\alpha \cap M^+([0,T] \times \T^2)$ (instead of just $\calC_\s^\alpha$). By the above multiplication proposition \ref{pr.multPOS} together with Schauder's estimate  \ref{th.SchauderP} and Lemma \ref{l.contract}, this map is now well defined from $\calC_\s^{\alpha+2 -\kappa}(\Lambda_t)$ to itself if we only suppose $\alpha+ 2 -\kappa>0$ when $\Theta\in \calC_\s^\alpha \cap M^+$ (as opposed to the previous condition $2\alpha + 2 - \kappa>0$).
Note that to prove the existence of a fixed-point for this map, only the continuity in the second argument of the Bilinear form from Proposition \ref{pr.multPOS} is used. 
This is the reason for the second threshold $\gamma_{\mathrm{pos}}= 2\sqrt{2} - 2>\gamma_{dPD}$ in Theorem \ref{th.main2} where the optimal regularity bound $\g2 - 2\sqrt{2} \gamma$ reaches $-2$.  
\qed




\section{Handling the punctures}\label{s.Ext}

To conclude our proof of Theorem \ref{th.mainT} (see also Corollary \ref{C.CoupConst} below), we need to explain how one can extend the analysis we just carried out for the simplified SPDE~\eqref{e.DLSs} to the SPDE we are interested in on the Torus $\T^2$, namely~\eqref{e.DLT}. Let us directly rewrite this SPDE after a Da Prato-Debussche change of variable as well as a time-change replacing $\frac 1 {4\pi}$ by $\frac 1 2$ in front of $\Delta$:
\begin{align*}\label{}
\begin{cases}
& \p_t v = \frac 1 {2}  \Delta v - \pi \mu \gamma  e^{\alpha_1 \gamma G_\tau(x_1,\cdot)} \Wick{e^{\gamma \Phi}} e^{\gamma v}  -R +  \frac {\alpha_1 \pi} { \lambda_{\hat g_\tau} (\T^2)}  \\
& v(0,\cdot) =  w
\end{cases}
\end{align*}
The exact same strategy as for the simplified SPDE can be used. Only two places need some more attention:
\bnum
\item One needs to show that for all $\gamma < \hat \gamma_c= 2\sqrt{2}$, and all $\alpha_1<Q$, the measure 
\[
e^{\alpha_1 \gamma G_\tau(x_1,\cdot)} \Wick{e^{\gamma \Phi}} 
\]
still belongs a.s. to $\calC_\s^\alpha$, for sufficiently small regularity $\alpha$, similarly (but with different bounds) as in Proposition \ref{pr.reg}. See Proposition \ref{pr.RegLine} below. 
\item If $\gamma<2$, we also need to show that $e^{\alpha_1 \gamma G_\tau(x_1,\cdot)} \Wick{e^{\gamma \Phi_\eps}}$ converges in probability to $e^{\alpha_1 \gamma G_\tau(x_1,\cdot)} \Wick{e^{\gamma \Phi}}$ in the above $\calC_\s^\alpha(\Lambda_t)$ spaces (similarly as in Proposition \ref{pr.conv}). 
\enum
Let us only discuss the first technical issue (the second one being of similar difficulty).

\subsection{$q^{\th}$ moments near the time-line singularity} $ $

In space-time ($z=(t,x)\in \R\times \T^2$),the log-singularity $G(x_1,x)\sim \log \frac 1 {\|x-x_1\|_2}$ induces a singularity in $\dist(z,\R_{x_1})^{-\alpha_1 \gamma} = \|x-x_1\|_2^{-\alpha_1 \gamma}$ (where $\R_{x_1}$ denotes the \textbf{time-line} $\{(t,x_1)\in \R\times \T^2, t\in \R\}$).
A first step in our analysis of the regularity of the weighted measure 
\[
\dist(z,\R_{x_1})^{-\alpha_i \gamma} \Wick{e^{\gamma \Phi(z)}}dz
\]  
is to obtain $q^{th}$ moments on the measure of $r$-balls in space-time which lie ON the singular time-line $\R_{x_1}$. 
\begin{proposition}\label{pr.momR}
Let us fix $\gamma<\hat \gamma_c = 2\sqrt{2}$ and $\alpha_1<Q$. Define 
\[
\xi_{\R}(q):= \g2(q-q^2) +(4 -  \alpha_1 \gamma) q\,.
\]
If $\sup_{1\leq q < 8/\gamma^2} \{ \xi_{\R}(q) -2\} >0$ (N.B. for some values of $\gamma$ this imposes a more restrictive condition on $\alpha_1$ than the Seiberg bound $\alpha_1<Q$), then for any 
\[
q<q^*(\gamma, \alpha_1):=\frac 8 {\gamma^2} \wedge\sup\{ q:
\xi_{\R}(q) -2=0\}  
\]
there exists $C<\infty$ s.t. for any $r\leq 1$ and any space-time point $z$ on the time-line $\R_{x_1}$, 
\begin{align*}\label{}
\sup_{\eps>0}\Eb{(\int_{B_\s(z,r)}  \dist(z,\R_{x_1})^{-\alpha_i \gamma}  \Theta_\eps)^q} \leq C r^{\xi_{\R}(q)}
\end{align*}
where $B_\s(z,r)$ is the $r$-ball around $z$ for the parabolic distance $\|\cdot\|_\s$. 
\end{proposition}

\ni
\textit{Proof.}
We will only sketch the main steps as the strategy is very similar to the proof of Proposition \ref{pr.mom} except the situation is less homogeneous due to the presence of the time-line $\R_{x_1}$. 

First one needs to prove a global integrability lemma analogous to Lemma \ref{l.T0}. 
\begin{lemma}\label{l.T0R}
For any $\gamma<\hat \gamma_c = 2 \sqrt{2}$ and any $(\alpha_1,q)$ which satisfy the same assumptions as in Proposition \ref{pr.momR}, 
\begin{align*}\label{}
\sup_{\eps>0} \Eb{(\int_{D=[0,1]\times[-1/2,1/2]^2} \dist(z,\R_{x_1})^{-\alpha_i \gamma}  \Theta_{\eps})^q} <\infty\,,
\end{align*}
where $D$ represents here a square which surrounds the time-line $\R_{x_1}$ (see Figure \ref{f.TimeLine}). 
\end{lemma}

\ni
\textit{Proof.} As in Lemma \ref{l.T0}, we argue using scaling arguments around the time-line $\R_{x_1}$. Instead of using a $1+2$-dimensional grid of time-spaces translates of $[0,\delta^2]\times [0,\delta]$, we only use a time-one-dimensional set of translates of $[0,\delta^2]\times[-\frac \delta 2, \frac \delta 2]$ as in Figure \ref{f.TimeLine}.
\begin{figure}[!htp]
\begin{center}
\includegraphics[width=0.65\textwidth]{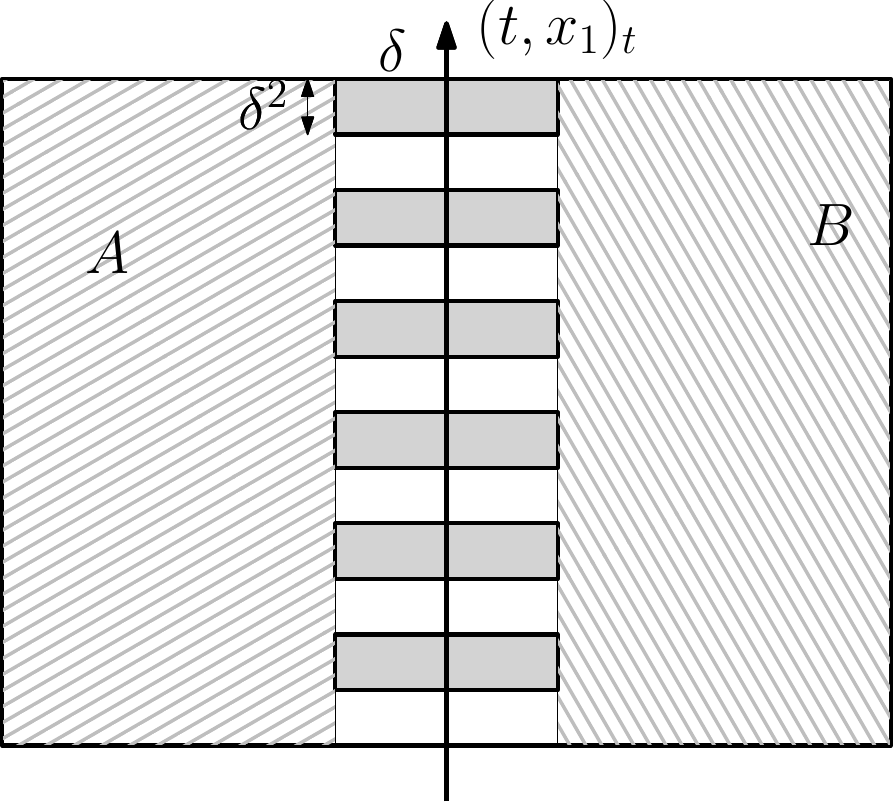}
\end{center}
\caption{}\label{f.TimeLine}
\end{figure}
We may use the triangle inequality to control the $q^{th}$ moments on parts $A$ and $B$ in Figure \ref{f.TimeLine} which we already know have finite $q^{th}$ moments. The $\delta^{-2}$ boxes along $\R_{x_1}$ are split as in Lemma \ref{l.T0} into two alternating sets $\calQ_1$ and $\calQ_2$ in order to use Kahane convexity inequality. We thus end-up for any $\eps, \delta>0$, with 
\begin{align*}\label{}
\Eb{(\int_{[0,1]\times[-1/2,1/2]^2}  \dist(z,\R_{x_1})^{-\alpha_i \gamma}  \Theta_{\delta \eps})^q} 
& \leq C \Eb{(\int_{\calQ_1}  \dist(z,\R_{x_1})^{-\alpha_i \gamma}  \Theta_{\delta \eps})^q}  + K\,,\\
\end{align*}
where $K$ is the uniform upper bound in $\eps$ which comes from the parts $A$ and $B$ in Figure \ref{f.TimeLine}. As in Lemma \ref{l.T0}, one works inductively on the integer $n$ s.t.  $n-1< q \leq n$ which allows us to avoid dealing with the non-diag terms in $\Eb{\Bigl(\sum_{Q\in \calQ_1}  (\int_Q e^{\gamma Z_{\delta \eps}^{Q}(z) - \g2 \calQ_{\eps}(0)})^{\frac q n} \Bigr)^n}$ (giving a larger value $K$ above). Only the diagonal-terms remain. There are $\delta^{-2}$ such terms (as opposed to $\delta^{-4}$ terms in the homogeneous case). Overall, we find 
\begin{align*}\label{}
&\Eb{(\int_{[0,1]\times[-1/2,1/2]^2}  \dist(z,\R_{x_1})^{-\alpha_i \gamma}  \Theta_{\delta \eps})^q} 
  \\
& \hskip 0.5 cm \leq C \delta^{4q-2} \delta^{\g2(q-q^2)} \delta^{-\alpha_i \gamma q}  \Eb{(\int_{[0,1]\times[-1/2,1/2]^2}  \dist(z,\R_{x_1})^{-\alpha_i \gamma}  \Theta_{\eps})^q}  + K\,,\\
\end{align*}
In order to conclude as in Lemma \ref{l.T0}, the crucial point is to have $\xi_{\R}(q)-2>0$ which follows from our set of assumptions. 
\qed

Now we deduce Proposition \ref{pr.momR} out of Lemma \ref{l.T0R} by scaling arguments (and Kahane inequality) exactly as we did in Proposition \ref{pr.mom}. 
\qed

\subsection{Handling $r$-balls which intersect the time-line}
Let $r\in(0,1]$. We wish to understand what is the order of magnitude of the largest mass (under the weighted measure $e^{\alpha_1 \gamma G(x_1,x)} \Theta(dz)$) among the $O(r^{-2})$ $\|\cdot\|_\s$ $r$-balls which intersect the time-line $\R_{x_1}$. For each of these squares and $q$ satisfying the conditions of Proposition \ref{pr.momR}, 
\begin{align*}\label{}
\Pb{\int_{B_\s(z,r)}  \dist(z,\R_{x_1})^{-\alpha_i \gamma}  d\Theta  \geq r^\beta} \leq r^{-\beta q} r^{\xi_\R(q)}\,.
\end{align*}
As there are $r^{-2}$ squares, one is looking for the largest possible exponent $\beta=\beta(\alpha_1,\gamma)>0$ s.t. 
\[
\sup_{0< q < \frac 8 {\gamma^2} \wedge\sup\{ q:
\xi_{\R}(q) -2=0\} } \{ \xi_\R(q) - \beta q -2 \}  >0 
\]

By a straightforward computation, it gives us that $\beta$ needs be smaller than 
\[
\bar \beta(\alpha_1,\gamma):= \g2 - (\alpha_1+2)\gamma + 4
\]
\margin{N.B. This is $\bar \beta = \g2 - 2\sqrt{2}\gamma + 4$ in the bulk}
(Note that since $\xi_\R(q) - 2 \geq \xi_\R(q) - \beta q -2$,  we have that if $\bar \beta(\alpha_1,
\gamma)>0$, then necessarily $\sup\{ q : \xi_\R(q) -2 \} >0$).  We conclude from the above estimates the following regularity property on the time-line $\R_{x_1}$.
\begin{lemma}\label{l.OnRline}
If $\gamma<\hat \gamma_c=2\sqrt{2}$ and $\alpha_1$ are such that $\bar\beta(\alpha_1,\gamma)>0$, then 
for any $\beta<\bar \beta(\alpha_1,\gamma)$, one has a.s. 
\begin{align*}\label{}
\sup_{z=(t,x)\in \R_{x_1}, r>0} \frac{\int_{B_\s(z,r)}  \dist(z,\R_{x_1})^{-\alpha_i \gamma}  d\Theta}
{r^{\beta}} & <\infty\,.
\end{align*}
\end{lemma}

\begin{remark}\label{}
As in the homogeneous case, this estimate is in fact optimal and shows that the measure $\mu(dz) := \dist(z,\R_{x_1})^{-\alpha_i \gamma}  \Theta(dz)$ does not belong a.s. to the regularity spaces $\calC_\s^\alpha$ if $\alpha \geq \bar \beta(\alpha_1,\gamma)-4 =  \g2 - (\alpha_1+2)\gamma$. 
\end{remark}

\begin{remark}\label{}
Let us make an important remark here: if one chooses to push $\alpha_1$ all the way to its Seiberg bound $\alpha_1=Q=\frac 2 \gamma + \frac \gamma 2$, then the regularity of the weighted measure  $\mu(dz) := \dist(z,\R_{x_1})^{-\alpha_i \gamma}  \Theta(dz)$ drops below $-2 - 2\gamma < -2$. As such the regularity is too small to apply our previous techniques (Section \ref{s.TORUS} as well as Section \ref{s.POS} which relies on the positivity). This is why we have less optimal conditions for example in Corollary \ref{C.CoupConst}. Note that weighted Besov norms in non-compact settings such as the one used in \cite{JcIsing} are likely to give slightly better thresholds here. 
\end{remark}

It remains to control the following $r$-balls: 
\subsection{Controlling the $r$-balls away from the time-line} 

Clearly, the condition will only consist in the simpler one $q<8/\gamma^2$ here as opposed to the balls intersecting the time-line in Lemma \ref{pr.momR}. We shall prove the following result. 
\begin{lemma}\label{l.distR}
For any $\gamma<\hat \gamma_c= 2\sqrt{2}$ and any  $\alpha_1$, there exists $C=C(\omega)<\infty$ a.s. such that for any $r\in(0,1]$ and any $r$-ball $Q$ (for the $\|\cdot\|_\s$-metric) at distance $\delta \geq r$ from the time-line, we have 
\[
\int_{B_r}  \dist(z,\R_{x_1})^{-\alpha_1 \gamma}  d\Theta \leq C \delta^{\tilde \beta - \alpha_1 \gamma } r^{\beta}\,,
\]
for any exponents satisfying $\tilde \beta < \frac{\gamma}{2\sqrt{2}}$ and $\beta< \bar \beta = \g2 - 2\sqrt{2}\gamma + 4$.  
\end{lemma}

\ni
\textit{Proof.}
Let us consider a grid of $r$-squares whose distance from $\R_{x_1}$ is in $[\delta,2 \delta]$, with $\delta \geq r$.  For each such $r$-square $Q$, one has for any $q<8/\gamma^2$, 
\begin{align*}\label{}
\Pb{\Theta(Q)  \geq r^{\beta} \delta^{\tilde \beta}}
& \leq r^{-\beta q} \delta^{-\tilde \beta q}  r^{\xi_\s(q)}
\end{align*}
\margin{We shall put the singularity $\delta^{- \alpha_1 \gamma q}$ later in the analysis}
Now, as there are $\delta r^{-4}$ such squares in this grid, we get the following by union-bound:  \margin{C: May 2018: SHOULDN'T THIS BE $\delta r^{-4}$ SQUARES ??   Probably this is what justifies $\hat q=2\sqrt{2}/\gamma$. \textbf{Checked, yes indeed, it must be 4!}. Also the notation $q^*$ was very misleading. CHECK what is is the use of $q^*$ in literature : $4/\gamma^2$ or rather ?? $2/\gamma$ ?}
\[
\Pb{\exists Q \text{ at dist. in $[\delta, 2\delta]$ s.t.  } \Theta(Q)  \geq r^{\beta} \delta^{\tilde \beta} } \leq O(1) \delta r^{-4} r^{-\beta q} \delta^{-\tilde \beta q} r^{\xi_\s(q)}
\]
When $\delta$ is fixed, we know what is the best possible exponent $\beta$: it must be less than $\bar \beta = \g2 - 2 \sqrt{2} \gamma +4$. Also, the optimal exponent in the bulk to run Kolmogorov's argument (see for example \cite{LBM} version 2) is $q=\hat q=\frac{2\sqrt{2}} \gamma$. Let us then choose any $\beta<\bar \beta$ and $q=\hat q$. Plugging these in the last displayed upper bound gives us 
\begin{align*}\label{}
\delta^{1-\tilde \beta \hat q} r^{\eps} = \delta^{1- \tilde \beta \frac{2\sqrt{2}} \gamma} r^\eps
\end{align*}
As such we see that we are free to take any $\tilde \beta < \frac{\gamma}{2\sqrt{2}}$ and Lemma \ref{l.distR} is proved.  (N.B. the proof is sub-optimal here and better exponents could be obtained by tuning $\hat q$ both as a function of $\gamma$ and also $\delta, r$). 
\qed

\subsection{Effects of the punctures on the regularity}

Combining Lemmas \ref{l.OnRline} and \ref{l.distR}, we thus obtain the following result. 
\begin{proposition}\label{pr.RegLine}
If $\gamma< 2\sqrt{2}$ and $\alpha_1<Q$ are such that $\bar \beta(\alpha_1,\gamma)>0$, then 
the measure 
$e^{\alpha_1 \gamma G_\tau(x_1,\cdot)} \Wick{e^{\gamma \Phi}}$ 
a.s. belongs to  $\calC_\s^\alpha$, for any 
\[
\alpha < \bigl[(\frac \gamma {2\sqrt{2}} - \alpha_1 \gamma)\wedge 0 \bigr] + \g2 - 2\sqrt{2} \gamma\,.
\]
\end{proposition}

\ni
\textit{Proof.} The proof is straightforward once one notices that the worse regularity is given by Lemma \ref{l.distR} as it is always the case that $ \bigl[(\frac \gamma {2\sqrt{2}} - \alpha_1 \gamma)\wedge 0 \bigr]  +\g2 - 2\sqrt{2} \gamma  \leq \g2  -(\alpha_1 + 2) \gamma$. \qed

Exactly as in Section \ref{s.TORUS}, such a regularity result easily  implies our Theorem \ref{th.mainP}. (The regularity thresholds in each case being $-1$ and $-2$). 
For the sake of simplicity, we will state a result below in the most important case of coupling constant $\alpha_1 := \gamma$, which corresponds to the conformal embedding of planar maps of $\gamma$-universality class. \margin{NOTE that if $\gamma\leq 1/(2\sqrt{2})=0.3535$ slightly below $\gamma_{dPD}\approx 0.38$, they have SAME regularity as $\wedge$ gives here 0!}

\begin{corollary}\label{C.CoupConst}
For any $\gamma<2$, fix the coupling constant $\alpha_1$ to be $\alpha_1(\gamma)= \gamma$. 
We have the following two regimes.
\bnum
\item If $\gamma< \gamma_1=-\frac 7 4 \sqrt{2} + \frac 1 4 \sqrt{130}\approx 0.376$ (thus slightly smaller than $\gamma_{dPD}\approx 0.378$ for the simplified SPDE), then  a local existence  as well as a convergence result (as $\eps\to 0$) holds for the  Liouville SPDE~\eqref{e.DLT}
\item If $\gamma_1 \leq \gamma <\frac 1 2 \sqrt{2} \approx 0.707$, only a strong solution to the SPDE~\eqref{e.DLT} holds similarly as in Theorem \ref{th.main2}. 
\enum
\end{corollary}


\margin{
Let us try a naive computation. If there is no push around $x_1$. Then, what is the local regularity ? Well, Clearly what is there is a \textbf{Brownian motion} and it is not going to have a biais typically! .... Good to explain in talks say ....}

%
%
%
%

\section{Dynamical liouville on the sphere $\S^2$}\label{s.sphere}

In this section we explain how to adapt the proof of Theorem \ref{th.mainT} to the case of the sphere (i.e. Theorem \ref{th.mainS}). We shall only focus on the simplified SPDEs on the sphere: 
\begin{align}\label{e.DLSsS}
\begin{cases}
\p_t X &  = \frac 1 {4\pi} \Delta X - e^{\gamma X} + \xi  \\
\text{ or } &  \\
\p_t X  & = \frac 1 {4\pi} \Delta X - \sinh(\gamma X) + \xi\,,
\end{cases}
\end{align}
where  $\Delta$ denotes the Laplace-Beltrami Laplacian on $\S^2$. Indeed, the extension to the actual Liouville SPDE~\eqref{e.DLS} can be done exactly as for the torus in Section \ref{s.Ext} (the fact there are 3 punctures ore more does not add any difficulty, only the largest coupling constant $\alpha_i$ matters in the estimates). As in the flat case, a key ingredient of the analysis will be a careful study of the linear heat equation on the sphere, 
\begin{align*}\label{}
\p_t \Phi = \frac 1 2 \Delta_{(\S^2)} \Phi  + \sqrt{2\pi} \, \xi\,.
\end{align*}
We will face the following three main technical difficulties due to the fact we are on the curved space $\S^2$ instead of the flat Euclidean setting $\R^2$:
\bnum
\item A first (small) issue is that we need to precise what we mean by the convolution $K*\xi$ (if $K$ is the Heat kernel operator on the sphere) as we are not on the Euclidean space. Similarly, we also need to be careful with associativity rules such as $K*(\rho_\eps * \xi)= K* \xi_\eps = K_\eps * \xi$ that are used all over the place in the study of singular SPDEs driven by space-time white noise. We will introduce precisely what will be our setup in subsection \ref{ss.setupS}

\item The main issue will be that there is no explicit expression known for the heat kernel $p_t^{\S^2}(x,y)$ on the two-sphere. We will therefore need some precise estimates to compare $p_t^{\S^2}(x,y)$ with $p_t^{\R^2}(x,y)$ in order to extract the needed sharp estimates on the correlation structure of the solution $\Phi$ to the linear heat-equation $\p_t \Phi = \frac 1 2 \Delta_{\S^2} \Phi  + \xi$.  See subsection \ref{ss.HKS}. 

\item Finally, we will explain briefly in Subsection \ref{ss.SpecS} how to adapt the analysis of the $q^{th}$ moments of $\Wick{e^{\gamma \Phi}}$ to the case of the sphere. 
\enum


\subsection{Setup/notations on the sphere $\S^2$}\label{ss.setupS}



$ $
First, the notations used in Subsection \ref{ss.Besov} to  introduce Besov spaces in the flat case have an obvious counter-part on the curved space/time $\R\times \S^2$. 
Let us just precise what shall be our parabolic metric (we won't have a norm anymore). For any $z=(t,x)\in \R\times \S^2$, let 
\[
\|z\|_\s = \|(t,x)\|_\s := |t|^{1/2} + d_{\S^2}(0,x)
\]
and 
\[
d_\s(z,z') = \|z-z'\|_\s := |t-t'|^{1/2} + d_{\S^2}(x,x')
\]
There is a slight abuse of notation here as $z-z'$ is no longer in $\R\times\S^2$.


\medskip

We now define what will be our kernel $K=K_{\S^2}$. As in Subsection \ref{ss.dPD} (and following \cite{sine}), it is more convenient here to consider a compactly supported (in space-time) kernel which coincides with the heat-kernel $p_t^{\S^2}(x,y)$ in the neighbourhood of $t=0$. As $\S^2$ is already compact, at first sight it seems one should not further restrict the support of $K$ in space. Yet,  because of the singular behaviour of the heat-kernel at the {\em cut-locus} of the sphere, it will in fact be convenient from now on to restrict our compact support to the north hemisphere. Here is a precise definition of our kernel $K$.

\begin{definition}\label{}
Let $K: (0,\infty) \times \S^2 \times \S^2 \to \R_+$ be a smooth kernel with the following properties:
\bi
\item $K_t(x,y)$ coincides with the heat-kernel $p^{\S^2}_t(x,y)$ when $ t \leq 1$ and $d_{\S^2}(x,y) \leq \frac \pi 3$. 
\item The support of $K$ is included in $\{t <2 \} \cap \{ d_{\S^2}(x,y) < \frac \pi 2 \}$. 
\ei
We also extend $K$ to negative times by defining it to be 0 on $((-\infty, 0] \times \S^2\times \S^2)\setminus( \{0\} \times \Delta)$, where $\Delta$ is the diagonal of $\S^2\times \S^2$. We then view $K$ as smooth kernel on $\R \times \S^2 \times \S^2 \setminus (\{0\} \times \Delta)$.

$K$ induces the following convolution type of operator. For any smooth function $f: \R \times \S^2 \to \R$, consider 
\begin{align}\label{e.K*}
K f(t,x) & := \int_0^\infty \int_{\S^2} K_{s}(x,y) f(t-s,y) ds dy \\
& = \int_{\R} \int_{\S^2} K_{s}(x,y) f(t-s,y) ds dy 
\end{align}
where $dy$ is the (non-normalized) Lebesgue measure on $\S^2$. 
\end{definition}


We also introduce the following $\eps$-smoothing of functions defined on the sphere. 

\begin{definition}[$\eps$-regularization on $\S^2$]\label{}
Fix some smooth and compactly supported function $\varrho: \R\times \R^2 \to \R_+$ which integrates to 1 and which is radially symmetric in space. (With a slight abuse of notation, we will write $\varrho(t,r e^{i \theta}) = \varrho(t,r)$). 
For any space-time distribution $Z$ on $\R \times \S^2$, define for any $\eps>0$, 
\begin{align*}\label{}
Z_\eps(x) = \varrho_\eps* Z(t,x) :=  \int_{\R \times \S^2} Z(s,y) \frac 1 {V_\varrho(\eps)}  \varrho(\frac {t-s} {\eps^2}, \frac {\dist_{\S^2}(x,y)} \eps ) ds dy\,,
\end{align*}
where $V_\varrho(\eps):= \int_{\R \times \S^2}  \varrho(\frac {s} {\eps^2}, \frac {\dist_{\S^2}(1,y)} \eps ) ds dy$. Note that  $V_\varrho(\eps)\sim \eps^4$ as $\eps\to 0$. See \cite{sphere, LecturesLQG} for a slightly different smoothing.  
\end{definition}

We may now define our (approximate) solution to the linear heat equation driven by $\xi_\eps:= \varrho_\eps * \xi$. We define $\Phi_\eps$ to be the Gaussian Field
\begin{align*}
\Phi_\eps: = K * \xi_\eps 
\end{align*}
As in \cite{sine}, it can be seen that the correction term to the heat equation $R_\eps := \p_t \Phi_\eps - \frac 1 2 \Delta \Phi_\eps - \xi_\eps$ is a smooth function which converges as $\eps\to 0$ to a smooth limiting  function $R$.  


We will use the following useful and straightforward property of the above $\eps$-smoothing (on $\R\times \R^2$, this is just the associativity of convolution).  \margin{Check it also on $K_\eps * R K_\eps$ }
\begin{lemma}\label{}
For any smooth function with compact support $f: \R \times \S^2 \to \R$,
\begin{align}\label{e.assoc}
K*f_\eps = K*(\varrho_\eps * f) & = K_\eps * f
\end{align}
where 
\[
K_\eps(t,0,x) =  (\varrho_\eps * K)(t,0,x) = \int_{\R \times \S^2} K(s,0,y) \frac 1 {V_\varrho(\eps)}  \varrho(\frac {t-s} {\eps^2}, \frac {\dist_{\S^2}(x,y)} \eps ) ds dy
\]
\end{lemma}
The proof is immediate. 

Let us now introduce the covariance kernel of the Gaussian field $\{ \Phi_\eps \}$ on which we will need to have precise asymptotics. 
\begin{definition}\label{}
Let $\calQ_\eps$ be the covariance kernel 
\begin{align*}\label{}
\calQ_\eps(t,x,y) := \Eb{\Phi_\eps(0,x) \Phi_\eps(t,y)}
\end{align*}
It is straightforward to check that it is given by the following  formula
\begin{align*}\label{}
\calQ_\eps(t,x,y) 
& = \int_\R \int_{\S^2} K_\eps(u,x,w) K_\eps(u-t,y,w)dudw\,.
\end{align*}
If $z=(t,x)$ is a space-time point in $\R \times \S^2$, we will call $\calQ_\eps(z):= \calQ_\eps(t,0,r e^{i \theta})$ if $r,\theta$ are the polar coordinates of $x\in \S^2$ viewed from the North-Pole. 
$\calQ_\eps$ can be written,  
\begin{align*}\label{}
\calQ_\eps
& = K_\eps * \calT(K_\eps)\\
& = \calQ * (\varrho_\eps * \calT \varrho_\eps)\, \;\;\;\;\;\;\; \calQ= K * \calT K\,,
\end{align*}
where $\calT f(t,x):=f(-t,-x)$ \margin{slight abuse of notations}
 and where the convolutions used here have the obvious meaning from~\eqref{e.K*} on the sphere. 
\end{definition}


\subsection{Estimates on the log-correlated field induced by the heat equation on the sphere}\label{ss.HKS}
$ $


As in Section \ref{s.RE}, we need to have a very precise control on the covariance kernel $\calQ_\eps$ of the approximate solution $\Phi_\eps$ to the heat-equation on $\S^2$. We will prove the following estimates which are an exact analog of the above Proposition \ref{l.key} in the case of the sphere.

\begin{proposition}\label{l.keyS}
The covariance kernels $\calQ$ and $\calQ_\eps$ satisfy the following properties on $\S^2$. 
\bnum

\item 
There exists a constant $\hat C_\varrho$ which only depends on the smoothing function $\varrho$ such that 
\begin{align}\label{e.Ess0S}
\calQ_\eps(0) (=\calQ_\eps(0,0,0)) = \log \frac 1 \eps +  \hat C_\varrho + O(\eps^2 \log \frac 1 \eps) 
\end{align}
N.B. We obtain an error term $O(\eps^2 \log \frac 1 \eps)$ on $\S^2$ instead of $O(\eps^2)$ on $\T^2$. 
\margin{!! J'ai rajoute une terme $\log \frac 1 \eps$ qui n'est pas dans Sine Gordon! A cause des estimées plus bas.. CHECK THIS IS FINE at least morally}


\item There exists $c,C$ which are $\varrho$-dependent s.t. for any $\eps\in (0,1]$ and any $z,\, \|z\|_\s \leq 1$, 
\begin{align}\label{e.Ess1epsS}
- \log(\| z\|_\s + \eps) + c \leq  \calQ_\eps(z) \leq - \log(\| z\|_\s + \eps) + C
\end{align}

\item  
\begin{align}\label{e.Ess1S}
\calQ(z) \sim - \log(\| z\|_\s)
\end{align}
\item 
\begin{align}\label{e.Ess2S}
|\calQ_\eps(z) - \calQ(z)| \leq C  \left( \frac \eps {\|z\|_\s} \wedge (1+ \log \frac \eps {\| z\|_\s} )\right)
\end{align}
\enum
%

\end{proposition}

\ni
{\em Proof.} 
These properties are already proved for the heat-equation on the Torus $\T^2$ in \cite{sine}. (See Proposition \ref{l.key} in section \ref{s.RE}).  We shall explain here how to adapt the proof to the case of the heat-equation on the sphere $\S^2$. The main technical difficulty to handle is the absence of explicit heat kernel $p_t^{\S^2}$ as well as the lack of exact scaling arguments.  We refer the reader to the Section 3 in \cite{sine} for the details of the proofs on the flat torus and will follow the same strategy/notations for the sphere $\S^2$ below. In the setting of the sphere, we still have the equality (in distributional sense) 
\[
\p_t K -\frac 1 2 \Delta K = \delta + R\,,
\]
where $\Delta$ is the Laplace-Beltrami Laplacian on $\S^2$ and $R$ is a smooth compactly supported function whose definition may change from line to line (as in \cite{sine}). We thus obtain in the same fashion as in \cite{sine} (noticing that $\calQ(z)=\calQ(-z)$ for space-time points $z$) that 
\[
\Delta \calQ(z) = K(z) + K(-z) + R\,,
\]
where in this setting if $z=(t, r e^{i \theta})$ (in polar coordinates from North pole, say), then $K(-z):=K(-t, re^{i \theta + i \pi})= K(-t, r e^{i \theta})$. 

As in Lemma 3.8. in \cite{sine}, we thus end up with the following useful expression of $\calQ$:
\begin{align}\label{e.Q0}
\calQ(t,x) & = (\hat K(t,\cdot) * G_{\S^2})(x) + R \nn \\
& := \hat\calQ(t,x) + R(t,x)  
\end{align}
where $\hat K(z):= K(z) + K(-z)$ and $G_{\S^2}(x)$ is the \textbf{Green function on the sphere}, i.e. the function $G_{\S^2}$ on $\S^2 \times \S^2$ defined (modulo an additive constant) for all $x\in \S^2$ by 
\[
- \Delta_{\S^2} G_{\S^2}(x,\cdot)  = 2\pi(\delta_x - \frac 1 {4\pi})\,.
\]
 
On the sphere, $G_{\S^2}$ takes the following simple explicit form 
\begin{align}\label{e.GreenSphere}
G_{\S^2}(x,y) = \log [ \frac 1 {2\sin(\frac r 2)}]\,,
\end{align}
where $r=d_{\S^2}(x,y)$. This can be easily checked using the following expression of the Laplace-Beltrami operator on $\S^2$ in the normal coordinate system from the north pole (see~\eqref{e.NC}): 
\begin{align}\label{e.DLB}
\Delta_{\S^2} f(r,\theta) = \frac 1 {\sin(r)} \frac {\p} {\p r}(\sin r \frac {\p f} {\p r}) + 
 \frac 1 {\sin(r)^2} \frac {\p^2 f} {\p \theta^2} 
\end{align}


We need to analyze 
\begin{align*}\label{}
(p^{\S^2}(|t|, \cdot) * G_{\S^2})(x)  & = \int_{\S^2} p^{\S^2}_{|t|}(0,y) G_{\S^2}(x,y) dy  (= \int_{\S^2} p^{\S^2}_{|t|}(x,y) G_{\S^2}(0,y) dy ) 
\end{align*}
As the Green function is explicit but the heat kernel is not, we will prefer the first expression.
\begin{align*}\label{}
(p^{\S^2}(|t|, \cdot) * G_{\S^2})(x)  & = \int_{\S^2} p^{\S^2}_{|t|}(0,y) \log [ \frac 1 2 \sin^{-1}(\frac {d(x,y)} 2)] dy 
\end{align*}
This expression is very convenient as it makes its $\p_x$ analysis much more amenable than if we had to analyze $\p_x p^{\S^2}_t(x,y)$ without having explicit expressions on the later one.

Let us first prove items (1) and the same estimate~\eqref{e.Ess1} as on the torus.
 Our strategy will be to compare our covariance kernel $\calQ$ with the corresponding kernel on the plane/torus. Indeed it is shown in \cite{sine}, using scaling arguments (that are not allowed in our present curved setting) that the latter kernel is very regular besides a $\log \frac 1 {\|z\|_\s}$ singularity (see Lemma 3.9 in \cite{sine} for a precise statement).
 
We shall use the following parametrization of the unit sphere $\S^2$ (the so called {\em normal coordinate system}):
\begin{align}\label{e.NC}
\begin{cases}
z&=\cos(r) \\
x&=\sin(r)\cos(\theta) \\
y&=\sin(r) \sin(\theta) \\
\end{cases}
\end{align}
and will denote by $(r, \theta)$ or $r e^{i \theta}$ the point at distance $r$ from the North pole with angle $\theta$.  Depending on the context, $r e^{i \theta}$ will either denote a point in $\S^2$ with the above parametrisation or the usual point in $\R^2$ in polar coordinates. We will need the following ``distortion'' Lemma between the polar coordinates on the sphere and the polar coordinates on the plane.

\begin{lemma}\label{l.distort}
There exists a constant $C>0$ s.t.
for any $x=r e^{i \theta},y=r' e^{i \theta'} \in \S^2$, 
\begin{align*}\label{}
 (1- C r^2) (1- C r'^2) \| r e^{i \theta} - r' e^{i \theta'}\|_{\R^2} 
\leq d_{\S^2}(x,y) \leq 
\| r e^{i \theta} - r' e^{i \theta'}\|_{\R^2}\,, 
\end{align*}
where on the L.h.s and R.h.s, $\| \cdot\|_{\R^2}$ denotes the Euclidean metric on the plane. 
\end{lemma}

{\em Proof:}
For any smooth path $\sigma \subset \R^2$ joining $\sigma(0)=a$ to $\sigma(t=1)=b$, the Euclidean length is given in polar coordinates on $\R^2$ by 
\[
L_{\R^2}(a,b) = \int_0^1 \sqrt{r_t^2 + r^2 \theta_t^2} dt
\]
while the Spherical length is given by
\[
L_{\S^2}(a,b) = \int_0^1 \sqrt{r_t^2 + \sin(r)^2 \theta_t^2} dt
\]
For paths $\sigma$ that remain at distance $r_0$ from the North pole, this readily implies that
\[
(1- O(r_0^2)) L_{\R^2}(a,b) \leq L_{\S^2}(a,b) \leq L_{\R^2}(a,b)
\]
This implies if $r \wedge r' \leq \pi/3$ (so that geodesics will remain in the ball of radius $r \wedge r'$) that 
\[
(1- O((r \wedge r')^2)) d_{\R^2}(r e^{i \theta}, r' e^{i \theta'}) \leq d_{\S^2}(x,y) \leq 
d_{\R^2}(r e^{i \theta}, r' e^{i \theta'})
\]
We conclude the proof by using $1 - C (r\wedge r')^2 \geq (1- C r^2) (1- C r'^2)$.
\QED

Recall that the kernel $K(t,x)$ is defined as $\chi(t,x) p_t^{\S^2}(x)$ where $\chi(t,x)$ is a smooth compactly supported function (and whose support in space is included in $\{ r e^{i \theta} \in \S^2, r < \pi/3 \}$).  

As mentioned earlier, there does not exist nice explicit expressions of the heat kernel $p_t^{\S^2}$. Nevertheless, in order to prove Lemma \ref{l.keyS}, we will need some rather precise analytical estimates on this heat kernel. We will use the following result from \cite{Nagase} which builds on earlier works such as \cite{Elworthy, ElworthyTruman, Ndumu} 
\begin{proposition}[\cite{Nagase}]\label{pr.nag}
For any $r \leq \frac \pi 3$,
\begin{align*}\label{}
p_t^{\S^2}(r) & = p_t^{\R^2}(r) e^{t/8} \bigl(\sqrt{\frac r {\sin(r)}} + O_\infty(t)\bigr) + O_\infty(e^{-1/t}) \\
& = p_t^{\R^2}(r) \bigl(\sqrt{\frac r {\sin(r)}} + O_\infty(t)\bigr) + O_\infty(e^{-1/t})\,.
\end{align*}
where $O_\infty(t)$ is a function of $(t,r)$ which is $C^\infty$ in $r$ and whose $k^{th}$ derivatives in $r$ are $O(t)$, uniformly in $r\leq \frac \pi 3$ and where $O_\infty(e^{-1/t})$ is a function, whose $l$ times derivative is uniformly bounded (on $\{ r \leq \frac \pi 3 \}$) by some $O(e^{-\eps_l /t})$ function for some $\eps_l>0$.  
\end{proposition}

Note that it is important for this analytical estimate to hold to be away from the {\em cut-locus} $M \subset \S^2 \times \S^2$. Since the log-singularity is integrable in $d=2$, we have (recall the definition of $\hat\calQ$ from~\eqref{e.Q0}) 
\begin{align*}\label{}
\hat\calQ(t,x)
&= \int_{\S^2} 
\chi(t,y) p^{\S^2}_{|t|}(0,y) \log [ \frac 1 2 \sin^{-1}(\frac {d(x,y)} 2)] \mathrm{Vol}_{\S^2}(dy)  \\
& = \int_0^{\pi/3}\int_0^{2\pi} \chi(t,re^{i \theta})  
p^{\S^2}_{|t|}(r) \log [ \frac 1 2  \sin^{-1}(\frac {d(x,y)} 2)]  \sin(r) dr d \theta
\end{align*}

\ni
\underline{\bf Proof of the lower bound in ~\eqref{e.Ess1S}.}

\medskip
This bound is easier, as $d_{\S^2}(x,y) \leq \| x -y\|_2$ which implies $\log[1/2 \sin^{-1}(d_{\S^2}(x,y)/2)] \geq \log[1/2 \sin^{-1}(\| x -y\|_2/2)]$. Furthermore, it turns out that for any $t>0$ and $r \geq 0$, we have 
\[
p_t^{\S^2}(r) \geq p_t^{\R^2}(r)\,.
\]
This is a consequence of the parabolic Harnack inequality in curved spaces. See for example \cite{Baudoin}.  
If one does not want to rely on this inequality, one can use instead the above Proposition \ref{pr.nag} and proceed as for the upper bound below.  We get 
\begin{align*}\label{}
\hat\calQ(t,x) 
& \geq \int_0^{\pi/3}\int_0^{2\pi} \chi(t,re^{i \theta})  
p^{\R^2}_{|t|}(r) \log [\frac 1 { 2 \sin(\| x- r e^{i \theta}\|_2/2)}]  \sin(r) dr d \theta \\
& \geq \int_0^{\pi/3}\int_0^{2\pi} \chi(t,re^{i \theta})  
p^{\R^2}_{|t|}(r) \log \frac 1 {\| x- r e^{i \theta}\|_2}  (r - O(r^3)) dr d \theta
\end{align*} 
The integral corresponding to integrating $rdr$ is exactly the kernel in the flat torus $\T^2$ which is analyzed by scaling arguments in Lemma 3.9  in \cite{sine}. It remains to justify that the remaining term corresponding to integrating $O(r^3) dr$ is negligible. It is not hard to check that 
uniformly in $x=r' e^{i \theta'}$ with $r' \leq \pi/3$, one has 
\begin{align}\label{e.rlogr}
\int_0^{\pi/3}\int_0^{2\pi} \chi(t,re^{i \theta})  
p^{\R^2}_{|t|}(r) \log \frac 2 {\| x- r e^{i \theta}\|_2}  r^3 dr d \theta 
& \leq O(1) \int_0^{\pi/3} p^{\R^2}_{|t|}(r) \log \frac 1 r   r^3   dr  \nn \\
& \leq O(1) \int_0^{\pi/3} \frac 1 {2\pi t} e^{-r^2/(2t)} \log \frac 1 r   r^3   dr \nn  \\
&  \leq O(1) |t| \log 1/|t| \nn \\
& \leq O(1) \|z\|_\s^{2} \log \frac 1 {\|z\|_\s}
\end{align} 
\margin{Check that estimate :) Seems fine. Idea is $\int_0^{\sqrt{t}}$ and analyse. This gives an upper bound $\int r^3 \log 1/r$ which by IPP gives $r^4 \log 1/r$ which gives what we want. On the other side, $r^2 > t$, so $\exp(-r^2 /t) \leq C t/r^2$ for example which is enough. }

Such an estimate on the deviation from $\calQ_{\T^2}$ together with Lemma 3.9 in \cite{sine} implies a lower bound for \eqref{e.Ess1S}. 
As in \cite{sine}, it easily transfers to the regularised kernel $\calQ_\eps$ and leads to a lower bound for \eqref{e.Ess1epsS}. 

\medskip
\ni
\underline{\bf Proof of the upper bound in ~\eqref{e.Ess1S}.}
\medskip

We now look for an upper bound on $\calQ$ close to $\calQ_{\T^2}$. 
We start by using Lemma \ref{l.distort} which gives us
\begin{align*}\label{}
\hat\calQ(t,x)
&=  \int_0^{\pi/3}\int_0^{2\pi} \chi(t,re^{i \theta})  
p^{\S^2}_{|t|}(r) \log [ \frac 1 {2 \sin(d(x,y)/2)}  ]  \sin(r) dr d \theta \\
& \leq 
 \int_0^{\pi/3}\int_0^{2\pi} \chi(t,re^{i \theta})  
p^{\S^2}_{|t|}(r) (\log [ \frac 1 {d(x,y)}] + \log[ 1 + O(d(x,y)^2)] ) \sin(r) dr d \theta \\
& \leq 
 \int_0^{\pi/3}\int_0^{2\pi} \chi(t,re^{i \theta})  
p^{\S^2}_{|t|}(r) [\log \frac 1 {\| x- r e ^{i \theta}\|_2} +  C (r^2 + r'^2) ]  r dr d \theta 
\end{align*}
where we used in the last inequality Lemma \ref{l.distort}, $\log(1/(1-x)) \leq C x$, $d(x,y)^2 \leq r^2 + r'^2$ and $\sin(r) \leq r$. 
Since $p_{|t|}^{\S^2}(\cdot)$ integrates to 1 on $\S^2$ and since $x=r' e^{i \theta'}$, the integral corresponding to the term $r'^2$ is bounded by $O(\|z\|_\s^2)$ while the integral corresponding to $r^2$, as for the lower bound in~\eqref{e.rlogr}, is bounded by $|t| \leq \|z\|_\s^2$. 

As such, up to a $\|z\|_\s^2$ correction term, it remains to bound from above,
\[ 
 \int_0^{\pi/3}\int_0^{2\pi} \chi(t,re^{i \theta})  
p^{\S^2}_{|t|}(r) \log \frac 1 {\| x- r e ^{i \theta}\|_2}\,. 
\]
We shall use for this the estimate form Proposition \ref{pr.nag} which gives us 
\begin{align*}\label{}
p_t^{\S^2}(r) 
& = 
p_t^{\R^2}(r) \bigl(\sqrt{\frac r {\sin(r)}} + O_\infty(t)\bigr) + O_\infty(e^{-1/t}) \\
& \leq p_t^{\R^2}(r)(1 + Cr^2 + C' t) +O(e^{1/t})  
\end{align*}
Using the same earlier computation~\eqref{e.rlogr} and the fact that $|t|\leq \|z\|_\s^2$, we obtain together with our previous lower bound that
\begin{align*}\label{}
\calQ(t,x) & = \hat \calQ(t,x) + R(t,x) \\
& = \int_0^{\pi/3}\int_0^{2\pi} \chi(t,re^{i \theta})  
p^{\R^2}_{|t|}(r) \log \frac 2 {\| x- r e ^{i \theta}\|_2} r dr d \theta  + O(\|z\|_\s^2 \log \frac 1 {\|z\|_\s}) + R(z) \\
& = \hat\calQ_{\T^2}(z)  + R(z) + O(\|z\|_\s^2 \log \frac 1 {\|z\|_\s})\,,
\end{align*}
if, say, one chooses the same approximation of identity $\chi(t,x)$ on $\T^2$ and $\S^2$.
\margin{I added a constant $\log 1/2$ in my definition of $G$ which is anyway up to constants to match more clearly with $\calQ_{\T^2}$.} 
Using Lemma 3.1 from \cite{sine}, we readily obtain the estimate~\eqref{e.Ess0S} with an error in $O(\eps^2 \log \frac 1 \eps)$ (as opposed to $O(\eps^2)$ in \cite{sine}). 

To finish the proof of Lemma \ref{l.keyS}, it only remains to establish the estimate~\eqref{e.Ess2S} which, as pointed out in Lemma 3.7. in \cite{sine}, follows easily from the Lemma stated below. \qed

\begin{lemma}\label{}
By denoting space-time points $z=(t,x)= (t, re^{i \theta})$ in the normal coord. syst.~\eqref{e.NC}, we have 
\begin{align}\label{e.Ess1L}
|\p_ r \calQ(z)| \leq C/ \|z\|_\s \;\;\;\;\; \text{      and      }  \;\;\;\;\;
|\p_t \calQ(z)| \leq C/  \|z\|_\s^2
\end{align}
which can be more concisely written as $|\p_i \calQ(z)| \leq C/ \|z\|_\s^{s_i}$. 
\end{lemma}

We shall use the expression (recall~\eqref{e.Q0})
\[
\hat \calQ(t,x) =
 \int_{\S^2} 
\chi(t,y) p^{\S^2}_{|t|}(0,y) \log [ \frac 1 2 \sin^{-1}(\frac {d(x,y)} 2)] \mathrm{Vol}_{\S^2}(dy) 
\]
It is clear one can differentiate under the integral sign here, which gives us the bound
\begin{align*}\label{}
|\p_r \calQ(t,x)| 
& \leq  C \int_{\S^2} \chi(t,y) p^{\S^2}_{|t|}(0,y) \frac 1  {d(x,y)}  \mathrm{Vol}_{\S^2}(dy) 
+ |\p_r R(z)| 
\end{align*}
From the properties of the heat kernel $p_t^{\S^2}$ we readily get that it is bounded by 
\[
O(1) \frac 1 {\sqrt{t} + d(0,x)} = O(1) \frac 1 {\|z=(t,x)\|_\s^1}\,.
\]
Now for the $\p_t$ derivative, instead of differentiating under the integral sign the spherical heat kernel $p_t^{\S^2}(0,r)$ (which would require unecessary detailed estimates), we use instead the identity 
\begin{align*}\label{}
\p_t \calQ(z) = \frac 1 2 \Delta \calQ + \text{(smooth. corr.)} \tilde R(z)\,.
\end{align*}
The same analysis as the above one (using, say, the expression of $\Delta_{\S^2} f$ from~\eqref{e.DLB}) shows that $
|\Delta \calQ(z)| \leq O(1) \frac 1 {\| z\|_\s^2}$ which concludes our proof.  \qed

\subsection{Multi-fractal spectrum of $e^{\gamma \Phi}$ on the sphere}\label{ss.SpecS}

Let us now briefly outline why the same moment bounds as in Proposition \ref{pr.moments} extend to the curved case of $\R\times \S^2$. (In particular, we get Proposition \ref{pr.conv} on the sphere). As in Section \ref{s.RE}, the main step is to obtain a uniform control of $q^{th}$-moments on macroscopic domains. For example the fact that for any $\gamma<\hat \gamma_c = 2 \sqrt{2}$ and any $q<8/\gamma^2$ (this is the direct analog of Lemma \ref{l.T0}), 
\begin{align*}\label{}
\sup_{\eps>0} \Eb{(\int_{D=[0,1]\times \S^2} \Theta_{\eps})^q} <\infty\,.
\end{align*}
The proof of this estimate follows the same lines as on the flat case (Kahane convexity inequality obviously holds in this setting as well). The only difference is that one cannot rely on an exact scaling argument as in Lemma \ref{l.KahBound}. To overcome this lack of exact scale invariance, one may rely on a suitably chosen $\delta$-grid covering $[0,1]\times \S^2$. As in the flat-case, the $\delta$-parabolic tiles will be $[0,\delta^2]$-long in time. To handle the curved nature of the space-variable, one may rely for example on a particularly  isotropic $\delta$-grid for the sphere, known as the \textbf{spherical geodesic grid}.  This gives a mostly-isotropic tiling of $\S^2$ made of $O(\delta^{-2})$ geodesic $\delta$-hexagons and only $12$ $\delta$-pentagons at macroscopic distance from each other (forming a icosahedron). Now, as $\delta \to 0$, each geodesic $\delta$-hexagon is very well approximated by an euclidean $\delta$-hexagon (with similar quantitative distortion bounds as in Lemma \ref{l.distort}). One can then compare the covariance of the field $\Phi_{\delta\eps}(z)$ on $\R\times \S^2$ with the flat case as in Lemma \ref{l.KahBound} up to an additive error of  $\log \frac 1 {1- C\delta^2}$ which follows from (the analog of) Lemma \ref{l.distort}. Then, this allows via Kahane convexity inequality to compare the $q^{th}$ moments in these geodesic $\delta$-hexagons with the $q^{th}$ moments in the flat case. Taking $\delta \to 0$ with careful uses of the triangular inequality for $L^q$, the proves the desired estimate. 

N.B. In fact as we have a lot of margin as $\delta\to 0$, instead of using the above spherical geodesic partition grid, one may as well rely on sufficiently many translates of unions of $\delta$-circles which altogether cover $\S^2$ (the point is that we do not need a partition of $\S^2$, some overlap is fine). This way, it is enough to apply directly Lemma \ref{l.distort}. 

These were the main differences between $\T^2$ and $\S^2$, the rest of the proof of Theorem \ref{th.mainS} proceeds exactly as in Sections \ref{s.TORUS} and \ref{s.RE} modulo the existence of a spherical Schauder estimate, which is the subject of the next remark.  
\begin{remark}\label{r.SchauderSphere}
As pointed to us by Nikolay Tzvetkov, a parabolic Schauder estimate such as Proposition \ref{th.SchauderP} should extend to the case of the sphere using the works \cite{Nikolay1, Nikolay2}.  For example the later paper constructs Besov spaces on manifolds without boundary and Proposition 2.5. from \cite{Nikolay2} shows that these Besov spaces are locally as in the flat Euclidean case (in quantitative terms). This should imply the desired parabolic regularity estimate (Proposition \ref{th.SchauderP}) in the case of the sphere $\S^2$. 
\end{remark}

\section{Comparison with Sine-Gordon and use of regularity structures}\label{s.RS}

To conclude the paper, let us make a comparison with the Sine-Gordon SPDE~\eqref{e.SG} and in particular the use of regularity structures in \cite{sine}. See also the more recent work \cite{CHS} which pushed the analysis all the way to $\beta^2=8\pi$.
\bi
\item We started by analyzing the ``Da Prato-Debussche'' phase $\gamma<\gamma_{dPD}=2\sqrt{2} -\sqrt{6}$ which corresponds to the regime $\beta^2 < 4\pi$ in Sine-Gordon. (This corresponds to Section 3 in \cite{sine}).  

\item In \cite{sine}, in order to push the existence of local solutions of~\eqref{e.SG} to higher values of $\beta$ (namely $\beta^2 < \frac {16} 3 \pi$), {\em second-order processes} are analysed and used within the framework of regularity structures from \cite{HairerReg}. These second-order processes correspond to re-centred versions of 
$\Psi_\eps (K * \Psi_\eps) = \Wick{e^{i\beta \Phi_\eps}} (K* \Wick{e^{i\beta \Phi_\eps}})$ (see sections 4,5 in \cite{sine}).  It seems very likely that a similar (tedious) analysis analogous to sections 4,5 in \cite{sine} (plus the fact one is lacking of arbitrary high moments here) should be doable in the Liouville case for an appropriate re-centred version of the second order processes $\Theta_\eps (K*\Theta_\eps)$. If so and once plugged into the regularity structures machinery, this would push the local existence for~\eqref{e.DLSs} (in a stronger sense than in Theorem \ref{th.main2}) for all $\gamma< \gamma^{(2)}$, where $\gamma^{(2)}$ is the smallest solution to $4 +3(\g2 - 2\sqrt{2} \gamma)=0$. This threshold $\gamma^{(2)}$ would be the analog of $\beta^{(2)}:=\sqrt{16\pi /3}$ in \cite{sine}. 

\item For Sine-Gordon, in order to push the local existence to higher values of $\beta$, higher-order processes would need to be analyzed. The suitable analysis of $k^{th}$-order processes should in principle enable to define the Sine-Gordon SPDE all the way to $\beta^{(k)}$, largest solution of 
\[
2k - (k+1) \frac {\beta^2} {4\pi} =0\,.
\]
We see here that $\beta^{(k)}$ indeed converges to $\beta_c=\sqrt{8\pi}$ as the order $k\to \infty$. In the Liouville case, one could then expect to define local solutions for~\eqref{e.DLSs} all the way to $\gamma^{(k)}$, smallest solution to 
\[
2k + (k+1) (\g2 - 2\sqrt{2} \gamma) =0\,.
\]
As the regularity function $\g2 - 2\sqrt{2} \gamma$ reaches regularity $-2$ precisely at $\gamma_{\mathrm{pos}}=2\sqrt{2} -2$, this means that $\gamma^{(k)} \to_{k\to \infty} \gamma_{\mathrm{pos}}$. Interestingly, this means that in our Theorem \ref{th.main2} (where we used the positivity of the non-linearity), we obtained the same threshold as what an arbitrary high-order regularity structure analysis would potentially give (with the drawback as mentioned earlier that we do not have a convergence result $X_\eps \to X$ in Theorem \ref{th.main2}, while such a stability result would follow from regularity structures). 
\ei

\begin{question}\label{q.c}
It is then an interesting problem to know whether $\gamma_{\mathrm{pos}}$ is the critical value above which the Liouville SPDE~\eqref{e.DLSs} would become super-critical, or whether local existence can be pushed yet to higher values of $\gamma$ by some other means. See also the discussion on intermittency and misleading subcriticality in Subsection \ref{ss.intermit}. 
\end{question}

\appendix
\section{Regularity estimates on Multiplicative chaos by direct moment computations}\label{a.1}

We give here a different proof for the (integer) polynomial moments of Gaussian multiplicative chaos which is not based on Kahane's convexity inequality. It is only a particular case of our Proposition \ref{pr.moments} but it has the advantage of being less technical and more importantly, it highlights the similarities/differences with Sine-Gordon (\cite{sine}). See also the appendix A in \cite{HairerQuastel} which deals with more general kernels. (In particular our case satisfies their assumption (A.2) and is thus covered by their results).

\begin{proposition}[Analog of Theorem 3.2 in \cite{sine}]\label{pr.Asine1}
For any $\gamma \in [0,2)$, any integer 
$1\leq N<N_c(\gamma) = \frac 8 {\gamma^2}$ and any $\kappa>0$ sufficiently small, one has
\begin{equ}[e.Nmoments1]
\Eb{\<{\phi_x^\lambda, \Theta_\eps}^N}  \lesssim \lambda^{-\g2 N(N-1)}
\end{equ}
\begin{equ}[e.Nmoments2]
\Eb{|\<{\phi_x^\lambda, \Theta_\eps - \Theta_{\bar \eps}}|^2}  \lesssim 
 (\eps\wedge \bar \eps)^{2\kappa} \lambda^{-2\kappa-{\gamma^2}}\;,
\end{equ}
uniformly over all test functions $\phi$ supported in the unit ball and bounded by 1, all $\lambda \in (0,1]$, and locally uniformly over space-time points  $x\in \R \times \T^2$ or $x\in \R \times \S^2$.  
\end{proposition}

%
%

\ni
{\em Proof.}
The second estimate follows almost verbatim from  \cite[Equation (3.12)]{sine}. Let us then focus on the first estimate. 

Since the distribution $\Theta_\eps \geq 0$ is not oscillatory, it makes the dependence on the test functions $\phi$ much easier to handle: for any test function $\phi$:
\begin{align*}\label{}
\Eb{\<{\phi_x^\lambda, \Theta_\eps}^N} 
& \leq \|\phi\|_\infty \lambda^{-4N} \Eb{(\int_{\Lambda(x,\lambda)} \Theta_\eps(x) dx)^N} \\
& \leq \lambda^{-4N} \int_{\Lambda(x,\lambda)^N}  {\prod_{1\leq i < j \leq N}  \exp(\gamma^2 \calQ_\eps(z_i-z_j))} 
dz     
\end{align*}
Using Lemma \ref{l.key}, this gives us the bound
\begin{align*}\label{}
\Eb{\<{\phi_x^\lambda, \Theta_\eps}^N} 
& \leq O(1) 
\lambda^{-4N} \int_{\Lambda(x,\lambda)^N}  \frac 1 {\prod_{1\leq i < j \leq N}  (\|z_i-z_j\|_\s+\eps)^{\gamma^2} } 
dz     
\end{align*}
\begin{remark}\label{}
Note here that as opposed to the {\em Coulomb-gas like} case in the proof of Theorem 3.2 in \cite{sine}, there is no vanishing numerator here which produces important cancellations. This absence of cancellations is due to a real-valued $\gamma$ instead of $i \beta$ and is the reason for the {\em intermittent behaviour.} 
\end{remark}

Now by making the following parabolic change of variable 
\[
z=(z_0,z_1,z_2) = \lambda\cdot w := (\lambda^2 w_0, \lambda w_1, \lambda w_2)
\]
we have that $\| \lambda\cdot w\|_\s = \lambda \| w \|_\s$  and we obtain the following upper bound:
\begin{align*}\label{}
\Eb{\<{\phi_x^\lambda, \Theta_\eps}^N} 
& \leq  \lambda^{-\g2 N(N-1) }\int_{\Lambda_0^N}  \prod_{1 \leq i < j \leq N} \frac 1 {(\| w_i - w_j\|_\s + \frac \eps \lambda)^{\gamma^2}} dw\,,
\end{align*} 
where $\Lambda_0$ denotes the parabolic ball $\Lambda(0,1)$ of radius 1. It remains to argue that this integral behaves well when $\gamma^2 < 4$ uniformly in $\eps>0$. This will follow readily from Lemma \ref{l.MI} below (which deals at once the uniformity in $\eps,\lambda$). 

%
%


\begin{lemma}\label{l.MI}
For all $0\leq \gamma < 2$ 
and all integer $1\leq N<N_c(\gamma) = \frac 8 {\gamma^2}$ 
\footnote{In the spatial case, the corresponding threshold on the exponent is $\frac 4 {\gamma^2}$}, there is a constant $C_{\gamma,N}<\infty$, such that 
\[
\int_{\Lambda_0^N}  \prod_{1 \leq i < j \leq N} \frac 1 {\| w_i - w_j\|_\s^{\gamma^2}} < C_{\gamma,N}\,.
\]
\end{lemma}

\begin{proof}
Let us first prove this Lemma for the unit  ball $\Lambda_0$ for the parabolic distance $\|\cdot\|_\s$ on $\R\times \R^2$, we will briefly discuss below how it adapts on $\S^2$. We proceed by induction on the integer $N\geq 1$. It is straightforward to check that the estimate holds when $N=1 < \frac 8 {\gamma^2}$.  Suppose now $2 \leq N < \frac 8 {\gamma^2}$ and assume the Lemma holds for $N-1$. 

Let us introduce the following quantity which will have nice scaling properties: for each $0<\eps<L$, define 
\begin{align*}\label{}
I_{\eps,L,N}:= \int_{
\substack{
w_2,...,w_N \in (\R^{d+1})^{N-1} \\
\inf_{i \neq j} \|w_i - w_j\|_\s \geq \eps \\
 \sup_i(\|w_i\|_\s) \leq L
}
}   
 \prod_{1 \leq i < j \leq N} \frac 1 {\| w_i - w_j\|_\s^{\gamma^2}} dw_2 \ldots dw_N\,,
\end{align*}
where $w_1$ denotes here the origin. 

Let us  analyze the behaviour of $I_{\eps,L,N}$ as $\eps \searrow 0$. For this, note that 
\begin{align}\label{e.split}
I_{\frac \eps 2,L,N}
& = I_{\frac \eps 2,\frac L 2,N} +  \int_{
\substack{
w_2,...,w_N \in (\R^{d+1})^{N-1} \\
\inf_{i \neq j} \|w_i - w_j\|_\s \geq \eps/2 \\
 \sup_i(\|w_i\|_\s) \in (L/2, L]
}
}   
 \prod_{1 \leq i < j \leq N} \frac 1 {\| w_i - w_j\|_\s^{\gamma^2}} dw_2 \ldots dw_N
\end{align}
The first term on the R.H.S is equal by scaling to $(1/2)^{4(N-1)} 2^{\gamma^2 \binom{N}{2}} I_{\eps,L,N}$. By our assmption on $N,\gamma$, we can find $\delta>0$ so that independently of $\eps, L$, 
\[
I_{\frac \eps 2,\frac L 2,N} \leq (1-\delta) I_{\eps,L,N}\,.
\]
This contraction type of bound will allow us to conclude. 
For the second term on the R.H.S of ~\eqref{e.split}, notice that after a possible reordering of points,  one can find $k\in \{1,\ldots, N-1\}$ s.t. the cloud of points $\{w_1=0,\ldots, w_k\}$ is at distance at least $L/(2N)$ from the cloud of points $\{w_{k+1}, \ldots, w_N\}$. 
The first cloud of points contributes at most $I_{\frac \eps 2,L,k}$ while the second contributes at most $I_{\frac \eps 2,2L,N-k}$. Both of these bounds are easily seen to 
be uniformly bounded from above by our recursion hypothesis. Now, the interaction between the two clouds is less than $(2N/L)^{\gamma^2 k (N-k)}$ and the combinatorial term coming from the reordering it at most $N!$.   
Recollecting, we find a constant $M=M_{\gamma, L,N} <\infty$ such that for all $\eps<L$, 
\begin{align*}\label{}
I_{\frac \eps 2,L,N} \leq (1-\delta) I_{\eps ,L,N} + M_{\gamma, L,N}
\end{align*}
which shows that $\limsup_{\eps \to 0} I_{ \eps ,L,N}<\infty$. 
By writing $\int_{\Lambda_0^N}  \prod_{1 \leq i < j \leq N} \frac 1 {(\| w_i - w_j\|_\s + \frac \eps \lambda)^{\gamma^2}} dw$ as an integral over $w_1$ of $N-1$-dimensional integrals and using a similar decomposition as above, we conclude the proof of the Lemma. 
For the case of $\R\times \R^2$, it is enough to use a cut-off $L$ large enough (for example $L=1$). In the case of the sphere $\S^2$,  the scaling argument does not have the same Jacobian everywhere. In that case it is of interest to use a cut-off $L=L(\gamma,N)$ sufficiently small so that one has $I_{\frac \eps 2,\frac L 2,N} \leq (1-\delta) I_{\eps,L,N}$ by approximate scaling. Then to get back to the integral in the Lemma, one can decompose as in ~\eqref{e.split}.
\end{proof}

\let\oldaddcontentsline\addcontentsline
\renewcommand{\addcontentsline}[3]{}%

\bibliographystyle{alpha}

\end{document}